\documentclass [12pt,english] {article}
\usepackage{amsthm}
\usepackage{amsmath}
\usepackage{amsfonts}
\usepackage{amssymb}
\usepackage[all]{xy}
\usepackage{color}   
\usepackage{tikz} \usetikzlibrary{arrows}

\usepackage{listings}
\usepackage{setspace}
\usepackage[font=small,labelfont=bf]{caption}
\singlespace
\def\TC{\protect\operatorname{TC}}

\def\zcl{\protect\operatorname{zcl}}

\def\hdim{\protect\operatorname{hdim}}

\newtheorem{ejem}{Example}
\newtheorem{defi}[ejem]{Definition}
\newtheorem{teo}[ejem]{Theorem}
\newtheorem{prop}[ejem]{Proposition}
\newtheorem{lema}[ejem]{Lemma}

\newtheorem{coro}[ejem]{Corollary}

\numberwithin{ejem}{section}

\begin{document}
\title{Bounds for the higher topological complexity of configuration spaces of trees}

\author{Teresa I. Hoekstra-Mendoza}

\date{\empty}

\maketitle

\begin{abstract}
For a tree $T$, we show that the higher topological complexity of the unordered configuration
space of $n$ points in $T$ is maximal for many values of $n$.

\end{abstract}

 \section{Introduction}
The concept of topological complexity “TC” was introduced by 
Farber motivated by the most basic problem of robot motion planning:
finding the smallest number of continuous instructions for a robot to
move from one point to another in a path-connected space.

 Rudyak extended Farber's notion of
topological complexity by defining, for $s \geq 2$, the $s$th topological complexity $TC_s (X)$ of a path-connected space $X$ which recovers Farber's original notion  for $s = 2.$

Given a topological space $X$, the  $s$-topological complexity $TC_s(X)$ can be described as 
  $TC_s(X)+1$ being the minimal number of open sets which cover the product $X^s$ such that on each open set, $e_s$ admits a section. Here $e_s: PX \rightarrow X^s$ is the sequential evaluation defined in Section 2.2.

The space $X$ is often viewed as the space of configurations of some real-world system. One example is when $X$ is the space of configurations of $n$ robots which move around a factory along a system of one-dimensional tracks. Such a system of tracks can be interpreted as a graph (a one-dimensional CW complex). 
  
For a finite graph $\Gamma$ and a positive integer $n$, let $\mathcal{C}^n\Gamma$ denote the configuration space of $n$ ordered points on $\Gamma$,
$$
\mathcal{C}^n\Gamma = \left\{(x_1,\ldots,x_n)\in \Gamma^n\colon x_i\neq x_j \mbox{ for } i\neq j \right\}.
$$
The usual right action of the $n$-symmetric group $\Sigma_n$ on $\mathcal{C}^n\Gamma$ is given by $(x_1,\ldots,x_n)\cdot \sigma=(x_{\sigma(1)},\ldots,x_{\sigma(n)})$, and $U\mathcal{C}^n\Gamma$ stands for the corresponding orbit space, the configuration space of $n$ unlabelled points on $\Gamma$. Both $\mathcal{C}^n\Gamma$ and $U\mathcal{C}^n\Gamma$ are known to be aspherical (\cite{A}); their corresponding fundamental groups are denoted by $P_n\Gamma$ (the pure $n$-braid group on $\Gamma$) and $B_n\Gamma$ (the full $n$-braid group or, simply, the $n$-braid group on $\Gamma$). We focus on the case of a tree $\Gamma=T$ (a tree is a connected graph which has no cycles).  

Our main goal is to study the topological complexity of the configuration spaces $\mathcal{C}^n T.$
 The topological complexity of these configuration spaces is related
to the number of essential vertices of $T$, which are the vertices of degree greater than or equal to three--more specifically, on how many vertices of degree exactly three there are in relation to the number of essential vertices of higher degree, and on how they are distributed along the tree $T$.
  
In \cite{S} Scheirer gave bounds for the usual ($s=2$) topological complexity of the configuration spaces of trees. Using slightly different techniques, we generalize Scheirer's results for any $s \geq 2$.
  
\  
  
We start section two with some terminology we shall use, as well as an introduction to discrete Morse theory and topological complexity. We also give a desription of Abrams' discrete model, and Farley and Sabalka's discrete gradient vector field, which are the tools we will use in the rest of the paper.  
  
In section three we give a description of how to obtain cup products in the cohomology ring of $U\mathcal{D}^nT$ for any tree $T$ and any $n \geq 4$, noticing that when the tree is binary, its cohomology ring is an exterior face ring.

In section four we prove our main results; we prove that the higher topological complexity or $U\mathcal{D}^nT$ is maximal for any tree $T$ and many values of $n$.

Finally in section five we compare our results in the special case when $s=2$ to Scheirer's results, and how to obtain at Scheirer's terminology starting from ours and viceversa.
   
\section{Preliminaries}
We start this section by introducing some notation and terminology which we shall use thoughout the paper.

Fix once and for all a planar embedding together with a root (a vertex of degree~1) for $T$. Order the vertices of $T$ as they are first encountered through the walk along the tree that (a) starts at the root vertex, which is assigned the ordinal 0, and that (b) takes the left-most branch at each intersection given by an essential vertex (turning around when reaching a vertex of degree 1). Vertices of $T$ will be denoted by the assigned non-negative integer. An edge of $T$, say with endpoints $r$ and $s$, will be denoted by the ordered pair $(r,s)$, where $r<s$. Furthermore, the ordering of vertices will be transferred to an ordering of edges by declaring that the ordinal of $(r,s)$ is $s$.  

Let $d(x)$ stand for the degree of a vertex $x$ of $T$, so there are $d(x)$ ``directions'' from $x$. For a vertex~$x$ different from the root, the direction from $x$ that leads to the root is defined to be the $x$-direction 0; $x$-directions $1, 2, \dots , d(x)-1$ (if any) are then chosen following the positive orientation coming from the planar embedding.  For instance, if $x$ is not the root and the vertex $y$ incident to $x$ in $x$-direction 0 is not essential (i.e.~$d(y)\leq2$), then $y=x-1$. Likewise, if $d(x)\geq2$, then $x+1$ is the vertex incident to $x$ in $x$-direction 1. It will be convenient to think of the only direction from the root vertex $0$ as 0-direction 1, in particular there is no $0$-direction 0.
Given an edge $e$ we shall denote by $ \iota(e)$ and $\tau(e)$ its endpoints with $\iota (e) < \tau(e)$.

Fix essential vertices $x_1<\cdots<x_m$ of T. The complement in $T$ of the set $\{x_1,\ldots,x_m\}$ decomposes into $1+ \sum_{i=1}^m \left(d(x_i)-1\right)$ conected components $C_{i,\ell_i}=C_{i,\ell_i}(x_1,\ldots,x_m)$, where $0\leq i\leq m$, $\ell_0=1$, and $1\leq\ell_i\leq d(x_i)-1$ for $i>0$. The closure of each $C_{i,\ell_i}$ is a subtree of $T$. $C_{0,1}$ is the component containing the root $0$, while $C_{i,\ell_i}$ (for $i>0$) is the component whose closure contains $x_i$ and is located on the $x_i$-direction $\ell_i$. The set $B(C_{i,\ell_i})$ of ``bounding'' vertices of a component $C_{i,\ell_i}$ is defined to be the intersection of the closure of $C_{i,\ell_i}$ with $\{x_1,\ldots,x_m\}$. Note that $x_i\in B(C_{i,\ell_i})$ for $i>0$, however the root $0$ is not considered to be a bounding vertex of $C_{0,1}$, just as no leaf of $T$ (i.e., a vertex of degree 1 other than the root) is considered to be a bounding vertex of any $C_{i,\ell_i}$. Furthermore, the set of pruned leaves is $L_{i,\ell_i}:=B(C_{i,\ell_i})\setminus\{x_i\}$. 

\begin{ejem}
Consider the tree of Figure \ref{comp}(left), we can see in Figure \ref{comp}(right) the connected components of $T-\{x_1,x_2,x_3\}$. The only non empty sets of pruned leaves are $L_{0,1} = \{x_1\}$ and $L_{1,2}=\{x_2, x_3\}$  
\end{ejem}

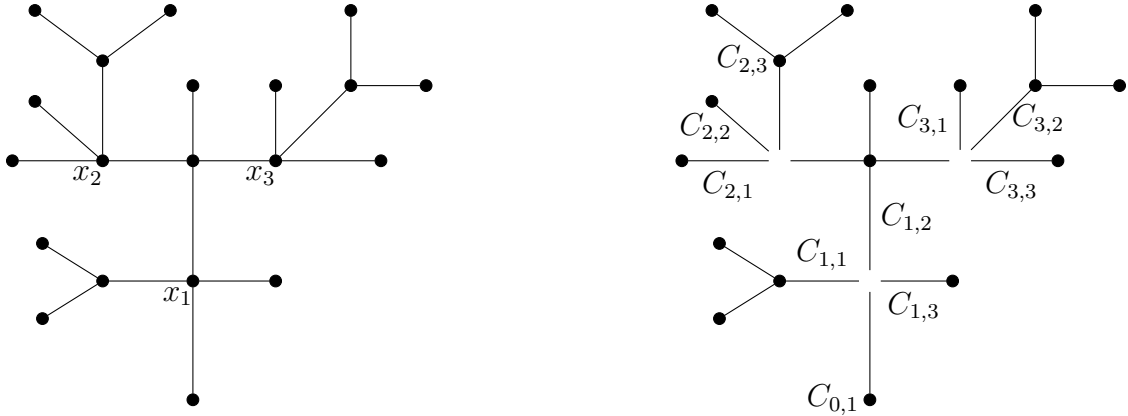
\begin{figure}[h!]
\begin{tikzpicture}[scale=1.0,rotate = 180, xscale = -1]

\node[circle, draw, scale=.4, fill=black] (1) at ( 3.5, 10.18) {};
\node (3) at ( 3.5, 8.6) {};
\node (4) at ( 2.3, 7) {};
\node[circle, draw, scale=.4, fill=black] (5) at ( 1.4, 6.21) {};
\node[circle, draw, scale=.4, fill=black] (6) at ( 1, 7) {};
\node[circle, draw, scale=.4, fill=black] (7) at ( 4.6, 8.6) {};
\node[circle, draw, scale=.4, fill=black] (8) at ( 1.5, 8.1) {};
\node[circle, draw, scale=.4, fill=black] (9) at ( 2.3, 8.6) {};
\node[circle, draw, scale=.4, fill=black] (10) at ( 2.3, 5.67) {};
\node[circle, draw, scale=.4, fill=black] (11) at ( 1.4, 5) {};
\node[circle, draw, scale=.4, fill=black] (12) at ( 3.2, 5) {};
\node[circle, draw, scale=.4, fill=black] (18) at ( 1.5, 9.1) {};
\node[circle, draw, scale=.4, fill=black] (19) at ( 3.5, 7) {};
\node[circle, draw, scale=.4, fill=black] (20) at ( 3.5, 6) {};
\node (21) at ( 4.7, 7) {};
\node[circle, draw, scale=.4, fill=black] (22) at ( 4.7, 6) {};
\node[circle, draw, scale=.4, fill=black] (23) at ( 5.7, 6) {};
\node[circle, draw, scale=.4, fill=black] (24) at ( 6, 7) {};
\node[circle, draw, scale=.4, fill=black] (25) at ( 5.7, 5) {};
\node[circle, draw, scale=.4, fill=black] (26) at ( 6.82, 6) {};

\node[circle, draw, scale=.4, fill=black] (1x) at ( -5.5, 10.18) {};
\node[circle, draw, scale=.4, fill=black] (3x) at ( -5.5, 8.6) {};
\node (v1) at (-5.7,8.8){$x_1$};
\node[circle, draw, scale=.4, fill=black] (4x) at ( -6.7, 7) {};
\node (v2) at (-6.9,7.2){$x_2$};
\node[circle, draw, scale=.4, fill=black] (5x) at ( -7.6, 6.21) {};
\node[circle, draw, scale=.4, fill=black] (6x) at ( -7.9, 7) {};
\node[circle, draw, scale=.4, fill=black] (7x) at ( -4.4, 8.6) {};
\node[circle, draw, scale=.4, fill=black] (8x) at ( -7.5, 8.1) {};
\node[circle, draw, scale=.4, fill=black] (9x) at ( -6.7, 8.6) {};
\node[circle, draw, scale=.4, fill=black] (10x) at ( -6.7, 5.67) {};
\node[circle, draw, scale=.4, fill=black] (11x) at ( -7.6, 5) {};
\node[circle, draw, scale=.4, fill=black] (12x) at ( -5.8, 5) {};
\node[circle, draw, scale=.4, fill=black] (18x) at ( -7.5, 9.1) {};
\node[circle, draw, scale=.4, fill=black] (19x) at ( -5.5, 7) {};
\node[circle, draw, scale=.4, fill=black] (20x) at (-5.5, 6) {};
\node (v3) at (-4.6,7.2){$x_3$};
\node[circle, draw, scale=.4, fill=black] (21x) at ( -4.4, 7) {};
\node[circle, draw, scale=.4, fill=black] (22x) at ( -4.4, 6) {};
\node[circle, draw, scale=.4, fill=black] (23x) at ( -3.4, 6) {};
\node[circle, draw, scale=.4, fill=black] (24x) at ( -3, 7) {};
\node[circle, draw, scale=.4, fill=black] (25x) at ( -3.4, 5) {};
\node[circle, draw, scale=.4, fill=black] (26x) at ( -2.4, 6) {};

\draw (3) -- (1)node[left]{$C_{0,1}$};
\draw (7) --node[below]{$C_{1,3}$} (3);
\draw (3) --node[above]{$C_{1,1}$} (9);
\draw (6) --node[below]{$C_{2,1}$} (4);
\draw (5) --node[left]{$C_{2,2}$} (4);
\draw (4) -- (10)node[left]{$C_{2,3}$};
\draw (11) -- (10);
\draw (10) -- (12);
\draw (9) -- (18);
\draw (8) -- (9);
\draw (19) --node[right]{$C_{1,2}$} (3);
\draw (20) -- (19);
\draw (21) -- (19);
\draw (22) --node[left]{$C_{3,1}$} (21);
\draw (23) --node[right]{$C_{3,2}$} (21);
\draw (24) --node[below]{$C_{3,3}$} (21);
\draw (19) -- (4);
\draw (23) -- (25);
\draw (26) -- (23);

\draw (3x) -- (1x);
\draw (7x) -- (3x);
\draw (3x) -- (9x);
\draw (6x) -- (4x);
\draw (5x) -- (4x);
\draw (4x) -- (10x);
\draw (11x) -- (10x);
\draw (10x) -- (12x);
\draw (9x) -- (18x);
\draw (8x) -- (9x);
\draw (19x) -- (3x);
\draw (20x) -- (19x);
\draw (21x) -- (19x);
\draw (22x) -- (21x);
\draw (23x) -- (21x);
\draw (24x) -- (21x);
\draw (19x) -- (4x);
\draw (23x) -- (25x);
\draw (26x) -- (23x);
\end{tikzpicture}
\caption{The tree $T$ and the connected components of $T-\{x_1, x_2, x_3\}$.}\label{comp}
\end{figure}
\subsection{Discrete Morse theory}
A very powefull tool for analyzing configuration spaces of graphs is discrete Morse theory, since it allows us to reduce the amount of cells in a complex while preserving the topological properties.
Assume throughout this section that $X$ is a regular complex.
\begin{defi}
A cell $\sigma$ of $X$ is called:
\begin{itemize}
\item critical, provided it does not appear as an entry of any pair of $W$;
\item redundant, provided there is a cell $\sigma'$ such that $(\sigma,\sigma')\in W$;
\item collapsible, provided there is a cell $\sigma'$ such that $(\sigma',\sigma)\in W$.
\end{itemize}
$W$ is called a discrete vector field if each cell of $X$ appears as an entry of at most one pair of $W$. In other words, $W$ is a discrete vector field provided any cell of $X$ is of one and only one of the three types above.
\end{defi}
For a redundant cell $\tau$ of $X$, we shall denote by $W(\tau )$ the unique cell of $X$ with $(\tau, W(\tau))\in W$.
\begin{defi}
Let $W$ be a discrete vector field on $X$. A sequence of $k$-cells, $\tau_1,\ldots,\tau_n$ satisfying $\tau_i \neq \tau_{i+1}$ for $ i =1,\ldots,n-1$ is called an upper $W$-path of length $n$ if, for each $i=1,\ldots,n-1$, $\tau_i$ is redundant and $\tau_{i+1}$ is a face of $W(\tau_i)$.
Similarly, if $\tau_i$ is collapsible with $\tau_i=W(\sigma_i)$, and $\sigma_i$ is a face of $\tau_{i-1}$ for $i=2, \dots ,n$ then the sequence is called a lower $W$-path of length $n$.
 The $W$-path is closed if $\tau_1=\tau_n$. We say that $W$ is a gradient vector field provided it does not admit closed $W$-paths.
\end{defi}

We can also think of the gradient paths as directed paths in the Hasse diagram as follows.
Let $X$ denote a finite regular cell complex. The Hasse diagram of $X$, $H_X$ is a directed graph, where the vertices are the cells of $X$, and there exists an arrow from the vertex $v$ to the vertex $w$ if $w$ is a face of $v$, and $\text{dim} (v)= \text{dim} (w)+1$. Given a discrete vector field $W$ on $X$, the modified Hasse diagram $H_X(W)$ is the directed graph obtained from $H_X$ by reversing every arrow belonging to $W$. A $W$-path is a directed path in $H_X(W)$ which alternates reversed arrows with arrows in $H_X$.

We need to recall how gradient paths recover (co)homological information. In the rest of the section we assume $W$ is a gradient field on $X$.

Start by fixing an orientation on each cell of $X$ and, for cells $a^{(p)}\subset b^{(p+1)}$, consider the incidence number $\iota_{a,b}$ of $a$ and $b$, i.e. the coefficient ($\pm1$, since $X$ is regular) of $a$ in the expression of $\partial(b)$. Here $\partial$ is the boundary operator in the cellular chain complex $C_*(X)$. The Morse cochain complex $\mathcal{M}^*(X)$ is then defined to be the graded $R$-free\footnote{Cochain coefficients are taken in a ground ring $R$, as we are interested in cup products.} module generated in dimension $p\geq0$ by the duals\footnote{We omit the use of an asterisk for dual elements.} of the oriented critical cells $A^{(p)}$ of $X$. The definition of the Morse coboundary map in $\mathcal{M}^*(X)$ requires the concept of multiplicity of upper/lower paths. For a path of length two, multiplicity is given by
\begin{equation}\label{multiplicitydefinition}
\mu(a_0\nearrow b_1\searrow a_1)=-\iota_{a_0,b_1}\cdot\iota_{a_1,b_1}\mbox{ \ \ and \ \ }
\mu(c_0\searrow d_1\nearrow c_1)=-\iota_{d_1,c_0}\cdot\iota_{d_1,c_1},
\end{equation}
and, in the general case, it is defined to be a multiplicative function with respect to concatenation of  paths. The Morse coboundary is then defined by
\begin{equation}\label{morsecoboundary}
\partial(A^{(p)})=\sum_{B^{(p+1)}}\left(\sum_{b^{(p)}\subset B} \left(\iota_{b,B} \sum_{\gamma\in\overline{\Gamma}(b,A)}\mu(\gamma)\right)\right)\cdot B.
\end{equation}
In other words, the Morse theoretic incidence number of $A$ and $B$ is given by the number of  gradient paths $\overline{\gamma}$ from $B$ to $A$ counted with multiplicity $\mu(\overline{\gamma}):=\iota_{b,B}\cdot\mu(\gamma)$.

\subsection{Abrams discrete model and Farley-Sabalka's gradient field}
Since the configuration space of graph is obtain by removing certain points from the product of the graph with itself, it does not have a CW-complex structure (nor a cubical or simplicial complex structure). This is why we use Abrams discrete model. For a tree $T$, think of $T^n$ as a cubical set. 

 Abrams discrete model for $\mathcal{C}^nT$ is the largest cubical subset $\mathcal{D}^nT$ of $T^n$ inside $\mathcal{C}^nT$. In other words, $\mathcal{D}^nT$ is obtained by removing open cubes from $T^n$ whose closure intersect the fat diagonal. As usual, the symmetric group $\Sigma_n$ acts on the right of $\mathcal{D}^nT$ by permuting factors. The action permutes in fact cubes, and the quotient complex is denoted by $U\mathcal{D}^nT$. 

Thus a cell in $U\mathcal{D}^nT$ can be written as $c=\{a_1, a_2, \dots, a_n\}$ where each $a_i$ is either a vertex or an edge of $T$, $a_i \cap a_j= \emptyset$ if $i \neq j$  and the dimension $c$ is  $ |\{ i \in \{1, \dots, n\}: a_i \in E(T)\}|$, that is,  the number of edges of $T$ that are in $c$. Each $a_i$ is called an ingredient of $c$ (vertex-ingredient or edge-ingredient).

\begin{teo}\cite{A} \label{teoa}
Let $\Gamma$ be a graph with at least $n$ vertices. Suppose
\begin{enumerate}
\item each path between distinct vertices of degree not equal to 2 in $\Gamma$ contains at least $n-1$ edges, and
\item each loop at a vertex in $\Gamma$ which is not homotopic to a constant map
contains at least $n+1$ edges.
\end{enumerate} 
Then, $\mathcal{C}^n\Gamma$ and $U\mathcal{C}^n\Gamma$ deformation retract onto  $\mathcal{D}^n\Gamma$ and $U\mathcal{D}^n\Gamma$,
respectively.
\end{teo} 

Theorem \ref{teoa} allows us to work with the space $U\mathcal{D}^nT$ which is much easier to work with than $U\mathcal{C}^nT$.

In \cite{FS} Farley and Sabalka gave a discrete gradient vector field for the space $U\mathcal{D}^nT$.
For a vertex $x$ of $T$ different from the root $0$, let $e_x$ be the unique edge of $T$ of the form $(y,x)$ with $y<x$. A vertex-ingredient $x$ of a cell $c$ is said to be blocked in $c$ if $x=0$ or, else, if $e_x \cap c \neq \emptyset$ and $x$ is said to be unblocked in $c$ otherwise. An edge-ingredient $e$ of a cube $c$ is said to be order-disrespecting in $c$ provided $e$ is of the form $(x,y)$ and there is a vertex ingredient $z$ in $c$ with $x<z<y$ and $z$ adyacent to $x$ (in particular $x$ must be an essential vertex); $e$ is said to be order-respecting in $c$ otherwise. Blocked vertex-ingredients and order-disrespecting edge ingredients in $c$ are said to be critical.
A cell is \textit{critical}  in Farley and Sablaka's gradient vector field if all of its ingredients are critical.
\subsection{Topological complexity}
For $s\geq2$, the $s$th topological complexity of a path-connected space $X$, $\TC_s(X)$, is defined as the sectional category of the evaluation map $e_s\colon PX\to X^s$ which sends a (free) path on $X$, $\gamma\in PX$, to the $s$-tuple$$e_s(\gamma)=\left(\gamma\left(\frac{0}{s-1}\right),\gamma\left(\frac{1}{s-1}\right),\ldots, \gamma\left(\frac{s-1}{s-1}\right)\right).$$ 

A standard estimate for the $s$th topological complexity of a space $X$, which we will use to obtain our bounds, is given by:
\begin{prop}\cite{B}\label{cota}
For a $c$-connected space $X$ having the homotopy type of a CW complex, 
$$ \zcl_s(X)\leq\TC_s(X)\leq s\hdim(X)/(c+1).$$
\end{prop}

The notation $\hdim(X)$ stands for the (cellular) homotopy dimension of $X$, i.e.~the minimal dimension of a CW complex having the homotopy type of $X$. On the other hand (and for our purposes), 
 the $s$th zero-divisor cup-length of $X$, $\zcl_s(X)$, is defined in purely cohomological terms\footnote{All cohomology groups in this paper are taken with $\mathbb{Z}$-coefficients.}. The zero-divisor cup-length, $\zcl_s(X)$ is the largest non-negative integer $\ell$ such that there are classes $z_j\in H^*(X^s)$, each with trivial restriction under the iterated diagonal inclusion $\Delta_s\colon X \hookrightarrow X^s$, and such that the cup product $z_1\cdots z_\ell\in H^*(X^s)$ is non-zero. Each such class $z_i$ is called an $s$th zero-divisor for $X$. The ``zero-divisor'' terminology comes from the observation that the map induced in cohomology by $\Delta_s$ restricts to the $s$-fold tensor power $H^*(X)^{\otimes s}$ to yield the $s$-iterated cup product.

\section{The cohomology ring of $U\mathcal{D}^nT$}
In \cite{GH} we analized the cohomology ring of $U\mathcal{D}^nT$, in particular we showed how to compute cup products. In view of proposition \ref{cota}, this is particularly useful for obtaining the bounds in section 4.

One of the most important aspects about configuration spaces of trees, is the following theorem, which allows to obtain cup products easily.

\begin{teo} \cite{df}
The Morse differential in $U\mathcal{D}^nT$ vanishes and therefore for each $m \geq 0$, a graded basis of $H^m(U\mathcal{D}^nT)$ is given by the cohomology classes of the duals of the critical $m$-cells.
\end{teo}

Assume $p=(p_1, \dots, p_r)$ and $q=(q_1, \dots, q_s)$ are integer vectors such that $p_i \geq 0 $ for $1 \leq i \leq r$, $q_i \leq 0 $ for $1 \leq i \leq s$ and there exists at least one $i \in \{1, \dots, r\}$ such that $p_i >0$.  Let $x$ be an essential vertex of degree $r+s+1$, we shall denote by $\{k|x,p,q\}$ the critical cell which has $k$ vertices blocked at the root vertex, the edge $(x, y)$ where $y$ is the vertex lying on the $x$-direction $r+1$, $p_i$ vertices blocked on the $x$-direction $i$ for $1 \leq i \leq r$, and $q_i$ vertices blocked on the $x$-direction $i+r$ for $ 1 \leq i \leq s$ (see Figure \ref{crit}).
Similarly we denote a critical $m$-cell by $\{k|x_1,p_1,q_1|x_2,p_2,q_2|\dots| x_m,p_m,q_m\}$ (see for example Figure \ref{3}).

\begin{figure}
\centering
\begin{tikzpicture}
\node[circle, draw, scale=.4, fill=black] (1) at (0,0){};
\node (01) at (0,0.3){$0$};
\node[circle, draw, scale=.4, fill=black] (2) at (1,0){};
\node[circle, draw, scale=.4, fill=black] (3) at (2,0){};
\node[circle, draw, scale=.4, fill=black] (4) at (4,0){};
\node  (04) at (3.8,0.3){$x$};
\node[circle, draw, scale=.4, fill=black] (5) at (5,0){};
\node[circle, draw, scale=.4, fill=black] (6) at (6,0){};
\node[circle, draw, scale=.4, fill=black] (7) at (4,1){};
\node[circle, draw, scale=.4, fill=black] (8) at (4,2){};
\node[circle, draw, scale=.4, fill=black] (9) at (4,-1){};
\node[circle, draw, scale=.4, fill=black] (0) at (4.7,.7){};
\node[circle, draw, scale=.4, fill=black] (11) at (5.4,1.4){};
\node[circle, draw, scale=.4, fill=black] (12) at (6.1,2.1){};
\node[circle, draw, scale=.4] (13) at (4.7,-0.7){};

\draw (1)--(2);
\draw (2)--(3);
\draw[dotted] (3)--(4);
\draw[very thick] (4)--(5);
\draw (5)--(6);
\draw (4)--(7);
\draw (7)--(8);
\draw (9)--(4);
\draw (0)--(4);
\draw (0)--(11);
\draw (11)--(12);
\draw[dotted] (4)--(13);
\end{tikzpicture}
\caption{The critical 1-cell $\{3|x,(2,3),(1,0,1)\}$}. \label{crit}
\end{figure}
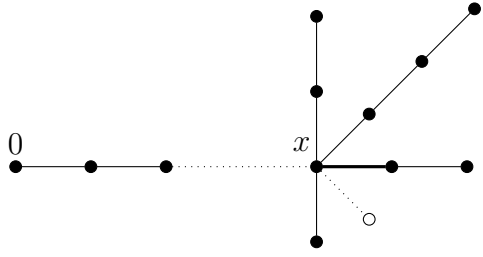

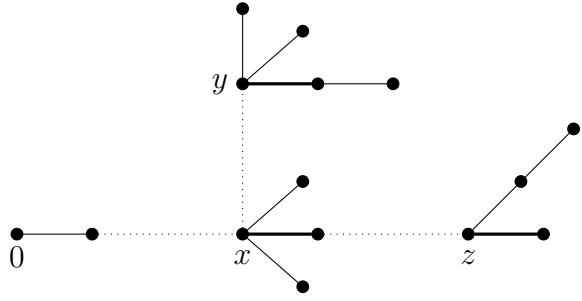
\begin{figure}
\centering
\begin{tikzpicture}
\node[circle, draw, scale=.4, fill=black] (1) at (0,0) {};
\node (01) at (0,-.3) {$0$};
\node[circle, draw, scale=.4, fill=black] (2) at (1,0) {};
\node[circle, draw, scale=.4, fill=black] (3) at (3,0) {};
\node (03) at (3,-0.3) {$x$};
\node[circle, draw, scale=.4, fill=black] (4) at (4,0) {};
\node[circle, draw, scale=.4, fill=black] (5) at (3.8,.7) {};
\node[circle, draw, scale=.4, fill=black] (6) at (3.8,-.7){};
\node[circle, draw, scale=.4, fill=black] (7) at (6,0) {};
\node (07) at (6,-0.3) {$z$};
\node[circle, draw, scale=.4, fill=black] (8) at (7,0) {};
\node[circle, draw, scale=.4, fill=black] (0) at (7.4,1.4){};
\node[circle, draw, scale=.4, fill=black] (11)at (6.7,.7){};
\node[circle, draw, scale=.4, fill=black] (12)at (3,2) {};
\node (012)at (2.7,2) {$y$};
\node[circle, draw, scale=.4, fill=black] (13)at (4,2) {};
\node[circle, draw, scale=.4, fill=black] (14)at (3,3) {};
\node[circle, draw, scale=.4, fill=black] (15)at (5,2) {};
\node[circle, draw, scale=.4, fill=black] (16)at (3.8,2.7){};
\draw (1)--(2);
\draw[dotted] (2)--(3); 
\draw[very thick] (3)--(4);
\draw (3)--(5);
\draw (3)--(6);
\draw[dotted](4)--(7);
\draw[very thick](7)--(8);
\draw (0)--(11);
\draw (7)--(11);
\draw[dotted] (3)--(12);
\draw[very thick] (12)--(13);
\draw (12)--(14);
\draw (13)--(15);
\draw (12)--(16);
\end{tikzpicture}
\caption{The critical 3-cell $c=\{2|x,(0,1),(0,1)|y,(1,1),(1)|z,(2),(0)\}$.} \label{3}
\end{figure}

Given the factors \begin{equation}\label{aproduct}
\left\{k_1| x_1,(p_{1,1},\ldots,p_{1,r_1}\hspace{-.3mm}),(q_{1,1},\ldots,q_{1,s_1}\hspace{-.3mm})\right\}{\cdots}
\left\{k_m|x_m,(p_{m,1},\ldots,p_{m,r_m}\hspace{-.3mm}),(q_{m,1},\ldots,q_{m,s_m}\hspace{-.3mm})\right\}
\end{equation} we are going to define the interaction parameters as follows, which will help us decide when a cup product of cells is non zero. 

\begin{defi}
The interaction parameters $\mathcal{R}_0$, $\mathcal{P}_i:=(\mathcal{P}_{i,1},\ldots,\mathcal{P}_{i,r_i})$ and $\mathcal{Q}_i:=(\mathcal{Q}_{i,1},\ldots,\mathcal{Q}_{i,s_i})$ of the factors in~$(\ref{aproduct})$ are given by
\begin{align*}
\mathcal{R}_0 & :=n\,+\!\sum_{x_j\in L_{0,1}}(k_j-n), \\
\mathcal{P}_{i,\ell_i} &:=p_{i,\ell_i}, + \sum_{x_j\in L_{i,\ell_i}}(k_j-n), \mbox{ for } i\in\{1,\ldots,m\} \mbox{ and } \ell_i\in\{1,\dots, r_{i}\},\mbox{ and }\\
\mathcal{Q}_{i,\ell_i} &:=q_{i,\ell_i}\,+ \sum_{x_j\in L_{i,\ell_i+r_{i}}}(k_j-n),\mbox{ for } i\in\{1,\ldots,m\} \mbox{ and }\ell_i\in\{1,\dots, s_{i}\}.
\end{align*}
If $\mathcal{R}_0\geq0$, $\mathcal{P}_i\geq0$ and $\mathcal{Q}_i\geq0$ for all $i=1,\ldots,m$, we say that the factors in~$(\ref{aproduct})$ interact weakly and, if in addition $\mathcal{P}_i>0$ for every $i$, we say that the factors in~$(\ref{aproduct})$ interact strongly. Otherwise, we say that the factors in~$(\ref{aproduct})$ do not interact.
\end{defi}

In the above definition $\mathcal{P}_i,\mathcal{Q}_i \geq 0$ means that every entry is greater than or equal to zero and $\mathcal{P}_i>0$ means that at least one entry is strictly greater than zero.

\begin{ejem}\label{para}
Assume we have the tree of Figure \ref{comp} sufficiently subdivided for $n=10 $ and consider the cells $ \{1|x_1,(1,7),(0) \}, \{7|x_2,(2,0),(0)\}$ and $ \{6|x_3, (1), (1,1)\}$. Recall that $L_{0,1} = \{x_1\}$ and $ L_{1,2}=\{x_2, x_3\}$.
Then the interaction parameters are $\mathcal{R}_0 = 10 + (9-10) = 1 \mathcal{P}_{1,2} = 7 + (7-10) + (6-10) =0$ and $\mathcal{P}_{i,j} = p_{i,j}$ and $ \mathcal{Q}_{i,j} = q_{i,j}$ for the remaining interaction parameters thus $ \mathcal{P}_{1,1} = 1,  \mathcal{Q}_{1,1}=0,  \mathcal{P}_{2,1}=2,  \mathcal{P}_{2,2}=0, \mathcal{Q}_{2,1}=0, \mathcal{P}_{3,1}=1, \mathcal{Q}_{3,1}=1$ and $\mathcal{Q}_{3,2}=1$.
\end{ejem}

\begin{prop}\cite{GH}\label{interactionproducts}
The product~$(\ref{aproduct})$ agrees with the critical $m$-cell $$\{\mathcal{R}_0| x_1,\mathcal{P}_{1},\mathcal{Q}_{1}|\cdots| x_m,\mathcal{P}_{m},\mathcal{Q}_{m}\}$$ provided the factors of~$(\ref{aproduct})$ interact strongly.
Recall $\mathcal{P}_i= (\mathcal{P}_{i,1}, \mathcal{P}_{i,2}, \dots , \mathcal{P}_{i,r_i})$ and $\mathcal{Q}_i= (\mathcal{Q}_{i,1}, \mathcal{Q}_{i,2}, \dots , \mathcal{Q}_{i,s_i})$.
\end{prop}

\begin{ejem}
Consider again the cells of Example \ref{para}. Since every interaction paramenter is greater or equal than zero, and, for every $i \in \{1,2,3\}$ there exists $j$ such that $P_{i,j} >0$ we have that the three factors interact strongly and thus by Proposition \ref{interactionproducts} their cup product is the 3-cell $\{1|x_1, (1,0),(0)|x_2,(2,0),(0)|x_3,(1),(1,1)\}$. 
\end{ejem}

\begin{teo}\cite{GH}\label{fact}
Any critical $m$-cell $\{\mathcal{R}_0| x_1,\mathcal{P}_{1},\mathcal{Q}_{1}|\cdots| x_m,\mathcal{P}_{m},\mathcal{Q}_{m}\}$ is the strong interaction product of $m$ critical 1-cells.
\end{teo}

Given a critical $m$-cell $c$ with edge-ingredients $(x_i, y_i)$ for $i = 1, \dots, m$, we can obtain one of its factors as follows: fix $i \in \{1, \dots, m\}$ and let $d_i$ be the cell obtained from $c$ by substituting every edge $(x_j, y_j)$ with the vertex $x_j$ for $j \neq i$. Then $d_i$ is a redundant 1-cell and the factor of $c$ is the unique critical 1-cell $c_i$ such that there exists a gradient path from $d_i$ to $c_i$.

To be specific, the cell $c_i$ is the critical 1-cell (which contains the edge-ingredient $(x_i,y_i)$) $\{k_i| x_i,p_i, q_i\}$ where $p_i =(p_{i,1}, \dots, p_{i,r})$, $q_i= (q_{i,1}, \dots, q_{i,s})$ and $p_{i,j}$ denotes the amount of vertex-ingredients in the cell $d_i$ that lie on $x_i$-direction $j$. Similarly $q_{i,j}$ denotes the amount of vertex-ingredients in the cell $d_i$ that lie on $x_i$-direction $j+r$. 

\

Now we want to see how a weak interaction product looks like. This case is a bit more complicated than the strong interaction product, since the weak interaction product yields, in many cases, a large sum of cells instead of only one cell.
Let $\Pi_1$ stand for a product~(\ref{aproduct}) whose factors interact strongly, so Proposition~\ref{interactionproducts} applies.
 Choose an additional 1-dimensional critical cell $\{k_x| x,(p_{x,1},\ldots,p_{x,r_x}),(q_{x,1},\ldots,q_{x,s_x})\}$  with $x<x_1<\cdots<x_m$ and where the standard conditions and conventions are assumed, namely, 
\begin{equation}\label{namely}
p_x:=(p_{x,1},\ldots,p_{x,r_x})>0 \ \ \mbox{and} \ \  q_x:=(q_{x,1},\ldots,q_{x,s_x})\geq0,
\end{equation}
where $r_x\geq1\leq s_x$, $r_x+s_x=d_x:=d(x)-1$, $|p_x|:=\sum_{\ell=1}^{r_x}p_{x,\ell}$, $|q_x|:=\sum_{\ell=1}^{s_x}q_{x,\ell}$, $k_x+|p_x|+|q_x|=n-1$ and $\overline{x}:=x[r_x+1]$. Consider the interaction parameters $P_{i}:=\mathcal{P}_{i}(x_1,\ldots,x_m)$ and $Q_{i}:=\mathcal{Q}_{i}(x_1,\ldots,x_m)$ of the factors of $\Pi_1$ ($i\in\{1,\ldots, m\}$), as well as the first three interaction parameters $R_0:=\mathcal{R}_0(x,x_1,\ldots,x_m)$, $P_x:=\mathcal{P}_1(x,x_1,\ldots,x_m)$ and $Q_x:=\mathcal{Q}_1(x,x_1,\ldots,x_m)$ of the factors of $\Pi_2:=\{k_x| x,p_x,q_x\}\cdot\Pi_1$.

\begin{teo}\label{productsvialowerpaths} \cite{GH}
In the situation above, if the factors of $\Pi_2$ interact but non-strongly, then
\begin{align}
\Pi_2&=-\sum_{a}\left\{R_0-|a|| x,a,Q_x| x_1, P_{1},Q_{1}|\cdots| x_m, P_{m},Q_{m}\right\}\label{laquefaltaba}\\ 
&\;\;\;\;{}+\sum_{\ell=1}^{s_x-1}\sum_{a,b}\left\{R_0-|a|-b-1| x,Q_x^{(\ell,a,b)},Q_x^{(\ell,+)}| x_1, P_{1},Q_{1}|\cdots| x_m, P_{m},Q_{m}\right\}\label{dosmas}\\
&\;\;\;\;{}-\sum_{\ell=1}^{s_x-1}\sum_{a,b}\left\{R_0-|a|-b| x,Q_x^{(\ell,a,b)},Q_x^{(\ell,-)}| x_1, P_{1},Q_{1}|\cdots| x_m, P_{m},Q_{m}\right\}\label{unomas}.
\end{align}
In the above equation, $a:=(a_1,\ldots,a_{r_x})$, $|a|:=a_1+\cdots+a_{r_x}$, 
$Q_x^{(\ell,+)}:=(Q_{x,\ell+1},\ldots,Q_{x,s_x}), $
$Q_x^{(\ell,-)}:=(Q_{x,\ell+1}-1,Q_{x,\ell+2},\ldots,Q_{x,s_x}) \mbox{ and }
Q_{x}^{(\ell,a,b)}:=(a_1,\ldots,a_{r_x},Q_{x,1}+b+1,Q_{x,2},\ldots,Q_{x,\ell}).$
The summation in~(\ref{laquefaltaba}) runs over all $r_x$-tuples $a$ of non-negative integer numbers satisfying $1\leq|a|\leq R_0$. The inner summation in~(\ref{dosmas}) runs over all $r_x$-tuples $a$ of non-negative integer numbers and all non-negative integer numbers~$b$ satisfying $|a|+b<R_0$. The inner summation in~(\ref{unomas}) is empty if $Q_{x,\ell+1}=0$, otherwise it runs over all $r_x$-tuples $a$ of non-negative integer numbers and all non-negative integer numbers $b$ satisfying $|a|+b\leq R_0$.
\end{teo}

\begin{ejem}\label{ej6}
Let $x$ be a vertex of degree five, in a tree $T$ and assume $y$ is a vertex of degree three that lies on $x$-direction 1. Take $n=9$, then $\{1|x,(2,0),(2,3)\} \smile \{7|y,(1),(0)\}= \{0|x,(0,1),(2,3)|y,(1),(0)\} + \{0|x,(1,0),(2,3)|y,(1),(0)\} + \{0|x,(0,0,3),(3)|y,(1),(0)\} - \{1|x,(0,0,3),(2)|y,(1),(0)\}-\{0|x,(0,1,3),(2)|y,(1),(0)\}-\{0|x,(1,0,3),(2)|y,(1),(0)\}$.
We can visiualize these six cells in Figure \ref{6}; the three cells of the top row have positive sign and the three cells of the bottom row have negatve sign.
\end{ejem}
\begin{figure}
\centering
\begin{tikzpicture}
\node[circle, draw, scale=.4] (1z) at (0,5){};
\node[circle, draw, scale=.4] (2z) at (0,6.5){};
\node[circle, draw, scale=.4] (3z) at (-2,8){};
\node[circle, draw, scale=.4, fill=black] (4z) at (-2.5, 8.5){};
\node[circle, draw, scale=.4] (5z) at (-1.5,8.5){};
\node (6z) at (-1,8){};
\node[circle, draw, scale=.4, fill=black] (7z) at (0,7.1){};
\node[circle, draw, scale=.4, fill=black] (8z) at (0,7.6){};
\node[circle, draw, scale=.4, fill=black] (9z) at (0,8.2){};
\node[circle, draw, scale=.4] (10z) at (0.4, 7){};
\node[circle, draw, scale=.4, fill=black] (11z) at (0.8,7.4){};
\node[circle, draw, scale=.4, fill=black] (12z) at (1.2,7.8){};
\node[circle, draw, scale=.4, fill=black] (13z) at (1.6,8.2){};

\draw (12z)--(13z);
\draw (11z)--(12z);
\draw (11z)--(10z);
\draw[very thick] (2z)--(10z);
\draw[dashed] (1z) -- (2z);
\draw[dashed] (2z)--(3z);
\draw[very thick] (3z)--(5z);
\draw (4z)--(3z);
\draw[dashed] (2z)--(6z);
\draw (2z)--(7z);
\draw (7z)--(8z);
\draw (8z)--(9z);

\node[circle, draw, scale=.4] (1xz) at (5,5){};
\node[circle, draw, scale=.4] (2xz) at (5,6.5){};
\node[circle, draw, scale=.4] (3xz) at (3,8){};
\node[circle, draw, scale=.4, fill=black] (4xz) at (2.5, 8.5){};
\node[circle, draw, scale=.4] (5xz) at (3.5,8.5){};
\node[circle, draw, scale=.4, fill=black] (6xz) at (4.5,7.5){};
\node[circle, draw, scale=.4] (7xz) at (5,7.1){};
\node[circle, draw, scale=.4, fill=black] (8xz) at (5,7.6){};
\node[circle, draw, scale=.4, fill=black] (9xz) at (5,8.2){};
\node[circle, draw, scale=.4, fill=black] (10xz) at (5.4, 7){};
\node[circle, draw, scale=.4, fill=black] (11xz) at (5.8,7.4){};
\node[circle, draw, scale=.4, fill=black] (12xz) at (6.2,7.8){};

\draw (11xz)--(12xz);
\draw (11xz)--(10xz);
\draw (2xz)--(10xz);
\draw[dashed] (1xz) -- (2xz);
\draw[dashed] (2xz)--(3xz);
\draw[very thick] (3xz)--(5xz);
\draw (4xz)--(3xz);
\draw (2xz)--(6xz);
\draw[very thick] (2xz)--(7xz);
\draw (7xz)--(8xz);
\draw (8xz)--(9xz);

\node[circle, draw, scale=.4] (1yz) at (10,5){};
\node[circle, draw, scale=.4] (2yz) at (10,6.5){};
\node[circle, draw, scale=.4] (3yz) at (8,8){};
\node[circle, draw, scale=.4, fill=black] (4yz) at (7.5, 8.5){};
\node[circle, draw, scale=.4] (5yz) at (8.5,8.5){};
\node (6yz) at (9,8){};
\node[circle, draw, scale=.4] (7yz) at (10,7.1){};
\node[circle, draw, scale=.4, fill=black] (8yz) at (10,7.6){};
\node[circle, draw, scale=.4, fill=black] (9yz) at (10,8.2){};
\node[circle, draw, scale=.4, fill=black] (10yz) at (10.4, 7){};
\node[circle, draw, scale=.4, fill=black] (11yz) at (10.8,7.4){};
\node[circle, draw, scale=.4, fill=black] (12yz) at (11.2,7.8){};
\node[circle, draw, scale=.4, fill=black] (13yz) at (9.25, 7){};

\draw (2yz)--(13yz);
\draw (11yz)--(12yz);
\draw (11yz)--(10yz);
\draw (2yz)--(10yz);
\draw[dashed] (1yz) -- (2yz);
\draw[dashed] (3yz)--(13yz);
\draw[very thick] (3yz)--(5yz);
\draw (4yz)--(3yz);
\draw[dashed] (2yz)--(6yz);
\draw[very thick] (2yz)--(7yz);
\draw (7yz)--(8yz);
\draw (8yz)--(9yz);

\node[circle, draw, scale=.4, fill=black] (1) at (0,0){};
\node[circle, draw, scale=.4] (2) at (0,1.5){};
\node[circle, draw, scale=.4] (3) at (-2,3){};
\node[circle, draw, scale=.4, fill=black] (4) at (-2.5, 3.5){};
\node[circle, draw, scale=.4] (5) at (-1.5,3.5){};
\node (6) at (-1,3){};
\node[circle, draw, scale=.4, fill=black] (7) at (0,2.1){};
\node[circle, draw, scale=.4, fill=black] (8) at (0,2.6){};
\node[circle, draw, scale=.4, fill=black] (9) at (0,3.2){};
\node[circle, draw, scale=.4] (10) at (0.4, 2){};
\node[circle, draw, scale=.4, fill=black] (11) at (0.8,2.4){};
\node[circle, draw, scale=.4, fill=black] (12) at (1.2,2.8){};

\draw (11)--(12);
\draw (11)--(10);
\draw[very thick] (2)--(10);
\draw[dashed] (1) -- (2);
\draw[dashed] (2)--(3);
\draw[very thick] (3)--(5);
\draw (4)--(3);
\draw[dashed] (2)--(6);
\draw (2)--(7);
\draw (7)--(8);
\draw (8)--(9);

\node[circle, draw, scale=.4] (1x) at (5,0){};
\node[circle, draw, scale=.4] (2x) at (5,1.5){};
\node[circle, draw, scale=.4] (3x) at (3,3){};
\node[circle, draw, scale=.4, fill=black] (4x) at (2.5, 3.5){};
\node[circle, draw, scale=.4] (5x) at (3.5,3.5){};
\node[circle, draw, scale=.4, fill=black] (6x) at (4.5,2.5){};
\node[circle, draw, scale=.4, fill=black] (7x) at (5,2.1){};
\node[circle, draw, scale=.4, fill=black] (8x) at (5,2.6){};
\node[circle, draw, scale=.4, fill=black] (9x) at (5,3.2){};
\node[circle, draw, scale=.4] (10x) at (5.4, 2){};
\node[circle, draw, scale=.4, fill=black] (11x) at (5.8,2.4){};
\node[circle, draw, scale=.4, fill=black] (12x) at (6.2,2.8){};

\draw (11x)--(12x);
\draw (11x)--(10x);
\draw[very thick] (2x)--(10x);
\draw[dashed] (1x) -- (2x);
\draw[dashed] (2x)--(3x);
\draw[very thick] (3x)--(5x);
\draw (4x)--(3x);
\draw (2x)--(6x);
\draw (2x)--(7x);
\draw (7x)--(8x);
\draw (8x)--(9x);

\node[circle, draw, scale=.4] (1y) at (10,0){};
\node[circle, draw, scale=.4] (2y) at (10,1.5){};
\node[circle, draw, scale=.4] (3y) at (8,3){};
\node[circle, draw, scale=.4, fill=black] (4y) at (7.5, 3.5){};
\node[circle, draw, scale=.4] (5y) at (8.5,3.5){};
\node (6y) at (9,3){};
\node[circle, draw, scale=.4, fill=black] (7y) at (10,2.1){};
\node[circle, draw, scale=.4, fill=black] (8y) at (10,2.6){};
\node[circle, draw, scale=.4, fill=black] (9y) at (10,3.2){};
\node[circle, draw, scale=.4] (10y) at (10.4, 2){};
\node[circle, draw, scale=.4, fill=black] (11y) at (10.8,2.4){};
\node[circle, draw, scale=.4, fill=black] (12y) at (11.2,2.8){};
\node[circle, draw, scale=.4, fill=black] (13y) at (9.25, 2){};

\draw (2y)--(13y);
\draw (11y)--(12y);
\draw (11y)--(10y);
\draw[very thick] (2y)--(10y);
\draw[dashed] (1y) -- (2y);
\draw[dashed] (3y)--(13y);
\draw[very thick] (3y)--(5y);
\draw (4y)--(3y);
\draw[dashed] (2y)--(6y);
\draw (2y)--(7y);
\draw (7y)--(8y);
\draw (8y)--(9y);
\end{tikzpicture}
\caption{The six 2-cells that appear in the product of Example \ref{ej6}.}\label{6}
\end{figure}
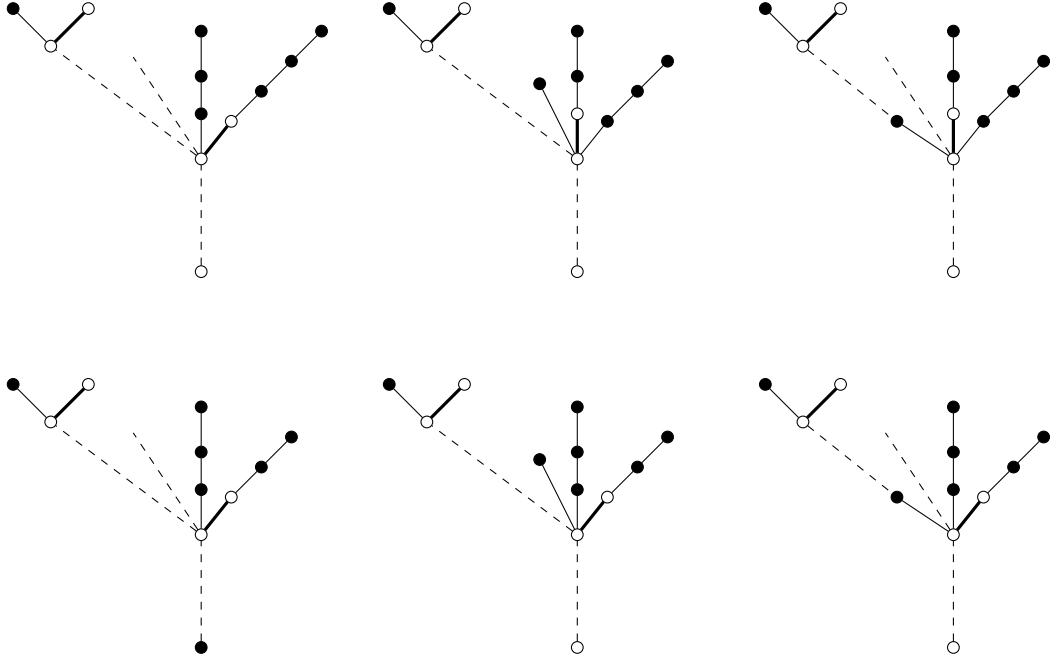
 
 The following corollary will be particularly useful in the next section.
\begin{coro}\label{remark}
Assume we have a product of cells $c_0 \smile \dots \smile c_m$ as in Theorem \ref{productsvialowerpaths} and let $c$ be a cell which appears as a summand of the product and let $d$ be the edge-ingredient of $c_0$ Then if $e$ is an edge-ingredient of $c$ and is not an edge-ingredient of $c_i$ for $0 \leq i \leq m$ then $\iota (e) = \iota (d) $ and $\tau (e)> \tau (d).$
\end{coro}
\begin{proof}
 By Theorem \ref{productsvialowerpaths}, since $c$ belongs to one of the three summands, and has $e$ as an edge-ingredient, it must have one of the following forms:
\begin{itemize} 
\item $\{R_0-|a|-b-1 |x,Q_x^{(\ell,a,b)},(Q_{x,\ell+1},\ldots,Q_{x,s_x})|\dots \}$
\item $\{R_0-|a|-b| x,Q_x^{(\ell,a,b)},(Q_{x,\ell+1}-1,Q_{x,\ell+2},\ldots,Q_{x,s_x})|\dots\}$
\end{itemize}
where $b$ is a non negative integer,
$a:=(a_1,\ldots,a_{r_x})$, $|a|:=a_1+\cdots+a_{r_x}$, and $
Q_{x}^{(\ell,a,b)}:=(a_1,\ldots,a_{r_x},Q_{x,1}+b+1,Q_{x,2},\ldots,Q_{x,\ell}).$
It is clear that $\iota (e) = \iota (d)$. Moreover, since the length of $Q_x$ is greater than the length of $(Q_{x,\ell+1},\ldots,Q_{x,s_x})$ and $(Q_{x,\ell+1}-1,Q_{x,\ell+2},\ldots,Q_{x,s_x})$ we must have that $\tau (e)> \tau (d).$
\end{proof}

\subsection{The simplicial complex $K_nT$}

In this section $T$ will be a binary tree, this is, a tree where every vertex has degree one, two or three.
Binary trees have a particularly nice cohomology ring.

 Let $$\langle k|x,p,q  \rangle = \sum_{i=1}^k \{k-i,x,p+i,q \},$$ notice that since $T$ is a binary tree, $p$ and $q$ are vectors of length one, thus we can see them as non negative integers.
\begin{defi}
Given a tree $T$, and an integer $n \geq 4$, let $K_nT$ denote the simplicial complex defined as follows:
\begin{itemize}
\item The vertices of $K_nT$ are the elements $\langle k|x,p,q \rangle$.
\item A set of $m+1$ vertices $\langle k_1|x_1,p_1,q_1 \rangle, \dots, \langle k_m|x_m,p_m,q_m \rangle$ is an $m$-simplex if the cells $\{ k_1|x_1,p_1,q_1 \}, \dots, \{ k_m|x_m,p_m,q_m \}$ interact strongly.
\end{itemize}
\end{defi}

\begin{ejem}
Consider the tree $T$ depicted in Figure \ref{K} (left) and assume it is sufficiently subdivided for $n=4$. Then  $K_4T$ is the simplicial complex obtained from the simplicial complex depicted in Figure \ref{K} (right) by adding  twelve isolated vertices. In $K_4T$ the three vertices of degree one are $\langle 2| v_2, (1,0) \rangle, \langle 0| v_2, (3,0) \rangle$ and $\langle 0| v_2, (1,2) \rangle$ and the three vertices of degree one are $\langle 2| v_3, (1,0) \rangle, \langle 2| v_4, (1,0) \rangle$ and $\langle 0| v_3, (3,0) \rangle$.
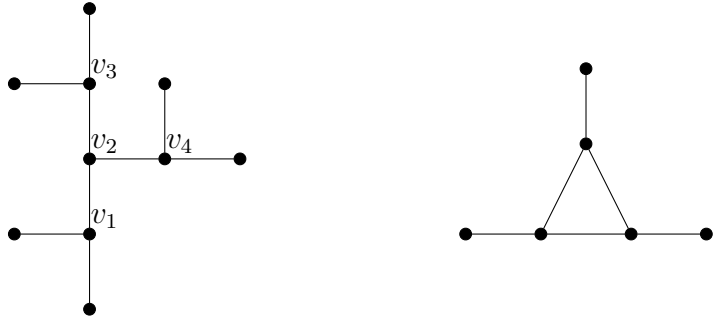
\begin{figure}
\centering
\begin{tikzpicture}
\node[circle, draw, scale=.4, fill=black] (*) at (0,0){};
\node[circle, draw, scale=.4, fill=black] (1) at (0,1){};
\node (10) at (0.2,1.2){$v_1$};
\node[circle, draw, scale=.4, fill=black] (2) at (-1,1){};
\node[circle, draw, scale=.4, fill=black] (3) at (0,2){};
\node (31) at (0.2,2.2){$v_2$};
\node[circle, draw, scale=.4, fill=black] (4) at (1,2){};
\node (41) at (1.2,2.2){$v_4$};
\node[circle, draw, scale=.4, fill=black] (5) at (2,2){};
\node[circle, draw, scale=.4, fill=black] (6) at (0,3){};
\node (60) at (0.2,3.2){$v_3$};
\node[circle, draw, scale=.4, fill=black] (7) at (1,3){};
\node[circle, draw, scale=.4, fill=black] (8) at (-1, 3){};
\node[circle, draw, scale=.4, fill=black] (9) at (0,4){};
\draw (3) to (6);
\draw (4) to (7);
\draw (*) to (1);
\draw (1) to (2);
\draw (1) to (3);
\draw (3) to (4);
\draw (4) to (5);
\draw (6) to (8);
\draw (6) to (9);

\node[circle, draw, scale=.4, fill=black] (a) at (5,1){};
\node[circle, draw, scale=.4, fill=black] (b) at (6,1){};
\node[circle, draw, scale=.4, fill=black] (c) at (7.2,1){};
\node[circle, draw, scale=.4, fill=black] (d) at (8.2,1){};
\node[circle, draw, scale=.4, fill=black] (e) at (6.6, 2.2){};
\node[circle, draw, scale=.4, fill=black] (f) at (6.6, 3.2){};
\draw (a) -- (b);
\draw (b) -- (c);
\draw (c) -- (d);
\draw (e) -- (f);
\draw (e) -- (b);
\draw (e) -- (c);
\end{tikzpicture}
\caption{The tree $T$ and part of the simplicial complex $K_4T$.}\label{K}
\end{figure}
\end{ejem}

In \cite{GH} we proved a slightly more general version of the following theorem.
\begin{teo}\label{pr} \cite{GH}
Assume $T$ is a binary tree. For a commutative ring $R$ with 1, the cohomology ring $H^{\ast}(B_nT;R)$
is the exterior face ring $\Lambda_ R(K_nT)$ determined by the simplicial complex $K_nT$ . Explicitly, $H^{\ast}(B_n T ; R)$ is the quotient $\Lambda/I$, where $\Lambda$ is the exterior graded $R$-algebra generated by the vertex set of $K_n T$, and $I$ is the ideal generated by monomials corresponding to non-faces of $K_nT.$
\end{teo}

\section{The s-topological complexity of $U\mathcal{D}^nT$}
In this section we are going to give conditions on $T$, depending on the value of $n$, which assure that the higher topological complexity of $U \mathcal{D}^nT$ is maximal. These conditions are related to the structure of the tree, and the relation between the number of vertices of degree tree and the number of vertices of degree higher than three.

\

Given a $1$-cell $a$, denote by $E(a)$ its unique edge-ingredient. The following proposition is a bit technical but it will be very helpful later on.

\begin{prop}\label{Gb}
Assume $a_1, \dots a_s$ are critical $m$-cells in $U\mathcal{D}^nT$ with no common edge-ingredient, and let $a_i = a_i^1 \smile \dots \smile a_i^m$ for $1\leq i \leq s$. Consider $b=\{ b_1, b_2 ,\dots,  b_s\}$, where each $b_i = \prod_{k=1}^s (\prod_{j \in I_{i,k}} a_i^j),$ and for each 
  $i \in \{ 1, \dots, s\}$, $\cup_{k=1}^s I_{i,k} = \{1,2, \dots, m\}$ and $ I_{1,j} \cap I_{i,k} = \emptyset$ when $j \neq k$. Then $a_i$ does not appear as a summand of $b_i$.
\end{prop}

\begin{proof}
In order to show that each $a_i$ does not appear as a summand  of the product $b_i$ for $i = 1, \dots ,s$, we are going to assume it does, and construct a graph $G_b$ associated to $b$ as follows.

  The vertices of $G_b$ are the edge-ingredients of the cells $a_i^j$, which we shall denote by $(i,j)$. The vertices $(i,j)$, which correspond to edge-ingredients of $a_i^j$ such that $a_i^j \in b_i$ (this is, $a_i^j$ is a factor of $b_i$), are isolated vertices in $G_b$. Two vertices $(i,j)$ and $(k,l)$ are adyacent in $G_b$ if $\iota (E(a_i^j)) = \iota (E(a_k^l))$.

Finally we shall give an orientation to every edge of $G_b$ as follows: take $((i,j),(k,l))$ an edge of $G_b$, we orient it from $(i,j)$ to $(k,l)$ if $\tau (E (a_i^j) > \tau (E (a_k^l))$.

Now we will show that every non-isolated vertex in $G_b$ has indegree and outdegree at least one. Recall Remark \ref{remark} and take a non-isolated vertex $(i,j)$. Since $a_i^j \notin b_i$, there exists $k \in \{1, \dots, s\}$, $k \neq i$ such that  $a_i^j \in b_k$. 
Because we are assuming $a_k \in b_k$, there exists $l$ such that $\iota (E (a_k^l ))= \iota (E(a_i^j))$ and $\tau (E(a_k^l)) > \tau (E(a_i^j))$. Thus $((k,l), (i,j))$ is an oriented edge in $ G_b$, and hence $(i,j)$ has indegree at least one. Now, since $a_i^j \notin b_i$, there exist $k', l'$ such that $a_{k'}^{l'} \in b_i$. This means that  $\iota (E (a_k^l ))= \iota (E(a_i^j))$ and $ \tau (E(a_i^j)) > \tau (E( a_{k'}^{l'}))$, thus $ ((i,j),(k',l')) \in G_b$. In turn, this means that $(i,j)$ has outdegree at least one.

 Thus every vertex in $G_b$ has in and out degree at least one. This implies that $G_b$ contains an oriented cycle. Let $\{(i_1, j_1), \dots, (i_r,j_r)\}$ be an oriented cycle in $G_b$ thus $((i_l, j_l),(i_{l+1}, j_{l+1}))$ and $((i_r, j_r), (i_1, j_1))$ are oriented edges in $G_b$ for $l \in \{1, \dots, r-1\}$ but this means that $$\tau (E(a_{i_1}^{j_1}) > \tau (E(a_{i_2}^{j_2}) > \dots >\tau (E(a_{i_r}^{j_r}) > \tau (E(a_{i_1}^{j_1}),$$ a contradiction.
\end{proof}

\begin{lema}\label{aristasdis}
Assume $a_1, \dots a_s$ are critical $m$-cells of maximal dimension in $U\mathcal{D}^nT$ with no common edge-ingredient. Then $TC(U \mathcal{D}^nT)  \geq sm$. 
\end{lema}
\begin{proof}
By Theorem \ref{fact}, the critial m-cells $a_i$ are the product of $m$ critical 1-cells thus $a_i= a_i^1 \smile \dots \smile a_i^m$ for $i= 1,\dots ,s$. Consider the following zero divisors:
$A_{i,j} = 1_1 \otimes 1_2 \otimes \dots \otimes 1_{i-1} \otimes a_i^j \otimes 1_{i+1}\otimes \dots \otimes 1_s - 1_1 \otimes 1_2 \otimes \dots \otimes 1_{i} \otimes a_i^j \otimes 1_{i+2}\otimes \dots \otimes 1_s$  and $A_{s,j} = 1 \otimes 1 \otimes \dots \otimes 1 \otimes a_s^j - a_s^j \otimes 1 \otimes \dots \otimes 1$ for $1 \leq j \leq m$ and $1 \leq i < s$.
We will show that the product 
\begin{equation}\label{0prod}
\prod_{i=1}^s (\prod_{j=1}^m A_{i,j})
\end{equation}
 is non zero.
 Notice that (\ref{0prod}) contains the terms $\pm a_1 \otimes a_2 \otimes \dots \otimes a_s$ and $ \pm a_s \otimes a_1 \otimes \dots \otimes a_{s-1}$. Since every cell $a_i$ is of maximal dimension, any product of more than $s$ factors is zero, thus any other non-zero term which appears in (\ref{0prod}) is of the form \begin{equation} b_1 \otimes b_2 \otimes \dots \otimes b_s  \mbox{, where each } b_i = \prod_{k=1}^s (\prod_{j \in I_{i,k}} a_i^j), \end{equation}\label{bb} and for each 
  $i \in \{ 1, \dots, s\}$, $\cup_{k=1}^s I_{i,k} = \{1,2, \dots, m\}$ and $ I_{1,j} \cap I_{i,k} = \emptyset$ when $j \neq k$.
By Proposition \ref{Gb}, $a_1 \otimes \dots \otimes a_s$ can not appear in any term of the form $b_1 \otimes \dots \otimes b_s$, thus the product (\ref{0prod}) is a non zero product of $ms$ zero divisors and by Proposition \ref{cota}, $TC_s(U\mathcal{D}^nT )\geq ms$.
\end{proof}

 Notice that if $v$ is a vertex of degree three then there exists a unique possible critical edge inciding in $v$. Given a vertex $v$ of degree three, the \textit{block} $B(v)$ of $v$ is the set consisting of the unique possible critical edge inciding in $v$ together with the vertex $v+1$.

\begin{teo}\label{par,s}
Let $m$ be the amount of essential vertices and $k$  the amount of vertices of degree three, thus $k \leq m$ in $T$.
Let $n = 2(m-k+l)+ \varepsilon$ with $\varepsilon\in \{0,1\}$ and $ l \leq \lfloor \frac{k}{s} \rfloor (s-1).$ Then $TC_s(U\mathcal{D}^nT) \geq s(m-k+l)$.
\end{teo}

\begin{proof}
We are going to exibit $s$ cells in $U\mathcal{D}^nT$ of maximal dimension which satisfy the hypothesis of Lemma \ref{aristasdis}. Assume first that $k\geq s$.

Let $x_1, \dots, x_{m-k}$ be the essential vertices of degree greater than three and let $y_{1}, \dots, y_k$ be the vertices of degree three. Assume for now that $\varepsilon=0$.
Assume that  $w_i, z_i, u_i \in N(x_i)$ for $i= 1, \dots m-k$, with $x_i<w_i<z_i<u_i$ and $k= sj+r$ with $ r \in \{0,1, \dots, s-1\}$. 
$$U_1=\{ B(y_1), \dots, B(y_j)\},U_2=\{ B(y_{j+1}), \dots, B(y_{2j})\}, \cdots, U_s= \{B(y_{(s-1)j+1}),\dots, B(y_{sj})\}$$ and $U_{s+1}=\{B(y_{sj+1}), \dots, B(y_{sj+r})\}$, in particular, $U_{s+1}$ can be an empty set.

 Let $ \Lambda= \cup_{i=1}^s U_i$, and $V_i = \Lambda-U_i$ for $1 \leq i \leq s$.
Notice that each $V_i$ has $(s-1)j\geq l$ blocks of the form $B(y_j)$, and that $\cap_{i=1}^s V_i = \emptyset$.

Let $W_i \subseteq V_i$ be such that $W_i$ contains exactly $l$ blocks of the form $B(y_j)$ for some $j$. Then the cells 
 $$c_i= \{ (x_1,u_1), z_1, (x_2,u_2),z_2, \dots, (x_{m-k},u_{m-k}), z_{m-k}, W_i\}$$ for $i$ even with $ 1< i \leq s$ and
 $$c_t= \{ (x_1,z_1), w_1, (x_2,z_2),w_2, \dots ,(x_{m-k},z_{m-k}), w_{m-k}, W_t \}$$ for $t$ odd with $1 \leq t \leq s$ are $s$ cells such that there is no edge belonging to all $s$ cells.
 If $\varepsilon= 1$, we take $c'_l = c_l \cup \{0\}$ (where $0$ denotes the root vertex) for $1 \leq l\leq s$. 
 Thus by Lemma \ref{aristasdis}, $TC_s(U\mathcal{D}^nT) \geq s(m-k+l)$.
 
 Consider now the case when $k <s$. Take $s' \leq k$ and construct $c_1, \dots, c_{s'}$ cells as above. Then $c_1^1, c_1^2, \dots, c_1^{s-s'}, c_2, \dots, c_{s'}$ are $s$ cells with no edge-ingredient in common, where $c_1^l =c_1$ for $ l \in \{1, \dots, s-s'\}$. Thus by Lemma \ref{aristasdis}, $TC_s(U\mathcal{D}^nT) \geq s(m-k+l)$.
\end{proof}

\begin{coro} 
 The bound obtained in Theorem \ref{par,s} is sharp
 \end{coro}
 \begin{proof}
 By Lemma \ref{cota}, $TC_s (U\mathcal{D}^nT) \leq s(h dim (U\mathcal{D}^nT))$ and notice that since $n= 2(m-\lfloor \frac{k}{s} \rfloor +l)+ \varepsilon$ with $ l \leq \lfloor \frac{k}{s} \rfloor (s-1)$,  we have $n < 2m$ and $h dim (U\mathcal{D}^nT)= \frac{n}{2}$ thus $TC_s (U\mathcal{D}^nT) \leq s\frac{n}{2}.$ 
 \end{proof}

So far, we have the case where there are many vertices of degree three in relation to the amount of vertices of degree three. On the other hand, when there are no vertices of degree three, thus the tree is binary, we know from Theorem \ref{pr} that the cohomology ring of $U\mathcal{D}^nT$ is an exterior face ring. This means that every critical $m$-cell has a unique factorization as the product of $m$ critical 1-cells, which is particular useful if we want to generalize lemma \ref{aristasdis}.

So, to combine these two results into a more general one, we shall construct a binary tree $O(T)$ which is relatively similar to $T$ which will help us construct zero divisors.

\

Given a tree $T$, we shall construct a binary tree $O(T)$ associated to $T$. Assume for now that $T$ has no vertices of degree two. Let $x_1, \dots, x_m$ be the essential vertices of $T$ and let $d_i$ denote $d(x_i)-1$ and let $\{y_i^0, \dots, y_i^{d_i}\}$ be the set of neighbours of $x_i$ in $T$.
\begin{itemize}
\item The vertex set of $O(T)$ is  $$V(O(T)) = \{ v \in V(T): d(v)=1\} \cup \bigcup_{i=1}^m\{x_i^1, x_i^2, \dots, x_i^{d_i-1}\}.$$
\item The edge set of $O(T)$ is $$E(O(T)) = \bigcup_{i=1}^m \left( \{(y_i^0,x_i^1), (x_i^{d_i-1},y_i^{d_i}),(x_i^{d_i-1},y_i^{d_i-1}) \}\bigcup_{j=1}^{d_i-2}  \{(x_i^j,y_i^j), (x_i^j, x_i^{j+1})\} \right).$$ 
\end{itemize}
In other words, we are substituting every vertex of degree $d >3$ with a path of length $d-2$ (see Figure \ref{bt}). Notice that for every vertex of degree three in $T$, there is a corresponding vertex in $O(T)$.
Note also that $O(T)$ does not depend on the choice of a plannar embedding of $T$, but ofcourse a plannar embedding of $T$ leads to a plannar embedding of $O(T)$.

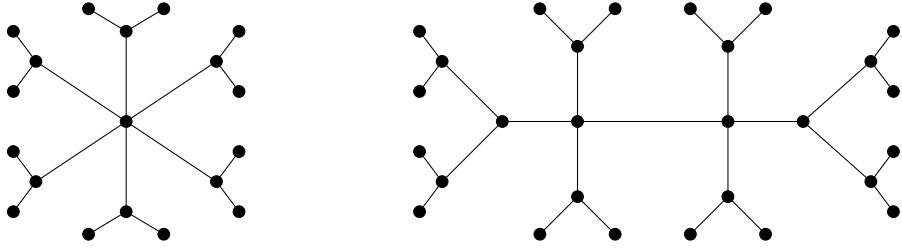
\begin{figure}
\centering
\begin{tikzpicture}
\node[circle, draw, scale=.4, fill=black] (0) at (0,0){};
\node[circle, draw, scale=.4, fill=black] (1) at (0,1.2){};
\node[circle, draw, scale=.4, fill=black] (1a) at (0.5,1.5){};
\node[circle, draw, scale=.4, fill=black] (1b) at (-0.5,1.5){};
\node[circle, draw, scale=.4, fill=black] (2) at (0,-1.2){};
\node[circle, draw, scale=.4, fill=black] (2b) at (0.5,-1.5){};
\node[circle, draw, scale=.4, fill=black] (2a) at (-0.5,-1.5){};
\node[circle, draw, scale=.4, fill=black] (3) at (1.2,.8){};
\node[circle, draw, scale=.4, fill=black] (3a) at (1.5,1.2){};
\node[circle, draw, scale=.4, fill=black] (3b) at (1.5,.4){};
\node[circle, draw, scale=.4, fill=black] (4) at (1.2,-.8){};
\node[circle, draw, scale=.4, fill=black] (4a) at (1.5,-.4){};
\node[circle, draw, scale=.4, fill=black] (4b) at (1.5,-1.2){};
\node[circle, draw, scale=.4, fill=black] (5) at (-1.2,.8){};
\node[circle, draw, scale=.4, fill=black] (6) at (-1.2,-.8){};
\node[circle, draw, scale=.4, fill=black] (5a) at (-1.5,1.2){};
\node[circle, draw, scale=.4, fill=black] (6a) at (-1.5,-1.2){};
\node[circle, draw, scale=.4, fill=black] (5b) at (-1.5,.4){};
\node[circle, draw, scale=.4, fill=black] (6b) at (-1.5,-.4){};
\draw (0)--(1);
\draw (1)--(1a);
\draw (1)--(1b);
\draw (0)--(2);
\draw (2)--(2a);
\draw (2)--(2b);
\draw (0)--(3);
\draw (3a)--(3);
\draw (3b)--(3);
\draw (4a)--(4);
\draw (4b)--(4);
\draw (0)--(4);
\draw (0)--(6);
\draw (0)--(5);
\draw (6)--(6a);
\draw (5)--(5a);
\draw (6)--(6b);
\draw (5)--(5b);

\node[circle, draw, scale=.4, fill=black] (3x) at (9.9,.8){};
\node[circle, draw, scale=.4, fill=black] (3ax) at (10.2,1.2){};
\node[circle, draw, scale=.4, fill=black] (3bx) at (10.2,.4){};
\node[circle, draw, scale=.4, fill=black] (4x) at (9.9,-.8){};
\node[circle, draw, scale=.4, fill=black] (4ax) at (10.2,-.4){};
\node[circle, draw, scale=.4, fill=black] (4bx) at (10.2,-1.2){};
\node[circle, draw, scale=.4, fill=black] (5ax) at (3.9,1.2){};
\node[circle, draw, scale=.4, fill=black] (6ax) at (3.9,-1.2){};
\node[circle, draw, scale=.4, fill=black] (5bx) at (3.9,.4){};
\node[circle, draw, scale=.4, fill=black] (6bx) at (3.9,-.4){};
\node[circle, draw, scale=.4, fill=black] (7) at (5,0){};

\node[circle, draw, scale=.4, fill=black] (8a) at (6.5,1.5){};
\node[circle, draw, scale=.4, fill=black] (8b) at (5.5,1.5){};
\node[circle, draw, scale=.4, fill=black] (8c) at (6.5,-1.5){};
\node[circle, draw, scale=.4, fill=black] (8d) at (5.5,-1.5){};
\node[circle, draw, scale=.4, fill=black] (9a) at (7.5,1.5){};
\node[circle, draw, scale=.4, fill=black] (9b) at (8.5,1.5){};
\node[circle, draw, scale=.4, fill=black] (9c) at (7.5,-1.5){};
\node[circle, draw, scale=.4, fill=black] (9d) at (8.5,-1.5){};
\node[circle, draw, scale=.4, fill=black] (9y) at (8,1){};
\node[circle, draw, scale=.4, fill=black] (9z) at (8,-1){};
\node[circle, draw, scale=.4, fill=black] (8z) at (6,1){};
\node[circle, draw, scale=.4, fill=black] (8y) at (6,-1){};
\node[circle, draw, scale=.4, fill=black] (8) at (6,0){};
\node[circle, draw, scale=.4, fill=black] (9) at (8,0){};
\node[circle, draw, scale=.4, fill=black] (10) at (9,0){};
\node[circle, draw, scale=.4, fill=black] (5x) at (4.2,.8){};
\node[circle, draw, scale=.4, fill=black] (6x) at (4.2,-.8){};

\draw (8a)--(8z);
\draw (8b)--(8z);
\draw (8c)--(8y);
\draw (8d)--(8y);
\draw (9a)--(9y);
\draw (9b)--(9y);
\draw (9c)--(9z);
\draw (9d)--(9z);

\draw (8)--(8y);
\draw (8)--(8z);
\draw (9)--(9y);
\draw (9)--(9z);
\draw (10)--(3x);
\draw (10)--(4x);
\draw (5x)--(7);
\draw(6x)--(7);
\draw (7)--(8);
\draw (8)--(9);
\draw (9)--(10);
\draw (6x)--(6ax);
\draw (5x)--(5ax);
\draw (6x)--(6bx);
\draw (5x)--(5bx);
\draw (3ax)--(3x);
\draw (3bx)--(3x);
\draw (4ax)--(4x);
\draw (4bx)--(4x);

\end{tikzpicture}
\caption{The trees $T$ and $O(T)$.}\label{bt}
\end{figure}

\begin{lema}\label{cam}
Let $T$ be a tree and consider the cells $a_i = \{k_i | x_i, p_i,q_i\} \in U\mathcal{D}^nT$ with $d(x_i)=3$ for $1\leq i \leq m$. With the notation introduced above, let $b_i = \{k_i |x_i^1, p_i,q_i\} \in U\mathcal{D}^nO(T)$ for $ 1 \leq i \leq m$. Then $a_1 \smile \dots \smile a_m \neq 0$ if and only if $b_1 \smile \dots \smile b_m \neq 0$. Moreover, $a_1 \smile \dots \smile a_m$ is a strong interaction product if and only if $b_1 \smile \dots \smile b_m$ is a strong interaction product, and if $a_1 \smile \dots \smile a_m = \{k | x_1, P_1,Q_1| \dots | x_m, P_m,Q_m\}$ then $b_1 \smile \dots \smile b_m= \{k | x_1^1, P_1,Q_1| \dots | x_m^1, P_m,Q_m\}$.
\end{lema}

\begin{proof}
The result follows from Proposition \ref{interactionproducts}, and the fact that the interaction parameters with respect to the factors $a_1, \dots, a_m$ and $b_1, \dots, b_m$ agree in both trees. 
\end{proof}

\begin{lema} \label{bueno}
Assume we have $ms$ critical 1-cells $a_i^j$ for $1 \leq i \leq s$ and $1 \leq j \leq m$ in $U
\mathcal{D}^nT$ with the following properties:
\begin{itemize}
\item There exists $k \in \{1, \dots, s\}$ such that for every $1 \leq i,l \leq s$ and $1 \leq j \leq k$, $E(a_i^j)= E(a_l^j)$ and $\iota (E (a_i^j))$ is a vertex of degree three. 
\item The product $ a_i^1  \smile \dots \smile a_i^m$ is a strong interaction product for every $1 \leq i \leq s$.
\item The $(m-k)s$ remaining edges are distinct, this is, $\bigcap_{i=1}^s \bigcap_{j=k+1}^m E(a_i^j) = \emptyset$. 
\item For every $j\in \{ 1,\dots, m\}$, there exist $i,l$ such that $ a_i^j \neq a_l^j$.
\end{itemize}
Then $TC_s(U\mathcal{D}^nT ) \geq sm$.
\end{lema}

\begin{proof}
Let $a_i = \langle a_i^1 \rangle \smile \langle a_i^2 \rangle \smile \dots \smile \langle a_i^k \rangle \smile a_i^{k+1} \smile \dots \smile a_i^m$ and let $\Lambda_{i,j}=1_1 \otimes \dots \otimes 1_{i-1} \otimes A_i^j \otimes 1_{i+1} \otimes \dots \otimes 1_s - 1_1 \otimes \dots \otimes 1_{i} \otimes A_i^j \otimes 1_{i+2} \otimes \dots \otimes 1_s$ for $ 1 \leq i \leq  s-1$ and $\Lambda_{s,j} = 1 \otimes \dots \otimes 1 \otimes a_s^j
- a_s^j \otimes 1 \dots \otimes 1$, where $A_i^j= \left\{ \begin{array}{rcl} a_i^j & \mbox{if} & j >k \\ \langle a_i^j \rangle & \mbox{if} & j \leq k \end{array} \right.$ 

We will show that the product $\prod_{i=1}^s \prod_{j=1}^m \Lambda_{i,j}$ is non zero. Notice that this product contains the term $\pm a_1 \otimes a_2 \otimes \dots \otimes a_s \pm a_s \otimes a_1 \otimes \dots \otimes a_{s-1}$. Let $a_i = c_i\smile d_i$ where $c_i = \langle a_i^1 \rangle \smile \langle a_i^2 \rangle \smile \dots \smile \langle a_i^k \rangle$ and $d_i =  a_i^{k+1} \smile \dots \smile a_i^m$. By Lemma \ref{cam} and Theorem \ref{pr} each cell $c_i$ has a unique factorization and thus can not be created using different factors. This means that any term $b=b_1 \otimes \dots \otimes b_s$ as in (\ref{bb}) containing the term $\pm a_1 \otimes \dots \otimes a_s$ must be such that each $b_i$ contains the term $c_i$ as a factor, otherwise we would obtain an other factorization of $c_i$ which is impossible. Thus $b_i= c_i \smile g_i$ for some $g_i$.
This means that $d_1 \otimes \dots \otimes d_s$ appears as a term of the product $g_1 \otimes\dots \otimes g_s$ which, by Proposition \ref{Gb} is impossible.
\end{proof}
In what remains of this section we shall assume $T$ is a tree with $m$ essential vertices and $k$ vertices of degree three thus $k \leq m$.
 
 For a vertex $x \in V(T)$ let $C_j(x)$ be the connected components of $T- \{x\}$ for $j=0, 1, \dots , d(x)-1$.

\begin{teo}\label{mayorj}
Let $n= 2(m-\lfloor \frac{k}{s} \rfloor +l)+ \varepsilon$ with $ l < \lceil \frac{k}{s} \rceil$ and $\varepsilon\in \{0,1\}$. Assume there exists a set $V$ of $l$ vertices having degree three such that: 
\begin{itemize}
\item The remaining $k-l$ vertices of degree three can be separated into  $s$ sets $U_1, \dots , U_s$ with $ \bigcap_{i=1}^s U_i=\emptyset$ where each set $U_i$ has cardinality $\lfloor \frac{(s-1)k}{s} \rfloor$ for $1 \leq i \leq s$.
\item For every $ y \in V$ there exists at least one $ j \in \{0,1,2\}$ and at least two indices $ i_1, i_2 \in \{1,2 ,\dots ,s \}$ such that $|C_j(y) \cap U_{i_1}| \neq |C_j(y) \cap U_{i_2}|$.
\end{itemize}
Then $TC_s(U\mathcal{D}^nT) \geq s(m-\lfloor \frac{k}{s} \rfloor +l)$.
\end{teo}

\begin{proof}
Assume first that $\varepsilon=0$.
Let $x_1, \dots x_{m-k}$ be the essential vertices of degree greater than three and let $y_{1}, \dots, y_k$ be the vertices of degree exactly three.
Let $B(U_i)= \{B(u): u \in U_i\}$ and $B(V)= \{B(v): v \in V\}$ for $ 1 \leq i \leq s$.  Assume that $w_i, z_i, v_i \in N(x_i)$ for $i= 1, \dots ,m-k$, with $x_i<w_i<z_i<v_i$.

Then $$c_i= \{ (x_1,v_1), z_1, (x_2,v_2),z_2, \dots, (x_{m-k},v_{m-k}), z_{m-k},B(V), B(U_i)\}$$ for $i$ even with $ 1< i \leq s$ and 
 $$c_j= \{ (x_1,z_1), w_1, (x_2,z_2),w_2, \dots, (x_{m-k},z_{m-k}), w_{m-k},B(V), B(U_j) \}$$ for $j$ odd with $1 \leq j \leq s$, are $s$ cells of dimension $m$.
If $\varepsilon=1$ we take the cells $c'_l = c_l \cup \{0\}$ for $ 1 \leq l \leq s$.  
 
By Theorem \ref{fact}, each cell $c_i$ is the strong interacion product of $m$ 1-cells $a_i^1, \dots, a_i^m$, where each factor is as described in the discussion following Theorem \ref{fact}.
 Let $a_i^j$ be the factor which contains the edge-ingredient $(x_j, v_j) $ or $(x_j, z_j)$ (if $i$ is even or odd respectively)  for $1 \leq j \leq m-k$, and let $a_i^{m-k+1}, \dots, a_i^{m-k+l}$ be the factors which contain an edge-ingredient inciding in a vertex of $V$, and $a_i^{m-k+l+1}, \dots, a_i^m$ be the factors which contain an edge-ingredient inciding in a vertex of $U_i$ for $ 1 \leq i \leq s$.
 
Notice first that for $j \in \{m-k+1, \dots, m-k+l\}$ and $i, g \in \{1, \dots, s\}$ we have that $E(a_i^j) = E(a_g^j)$ and $ \iota (E(a_i^j)) \in V$.

Now, since $\bigcap_{i=1}^s \bigcap_{j=1}^{m-k} E(a_i^j) = \bigcap_{i=1}^s (x_i, v_i) \cap (x_i, z_i) =\emptyset ,$ and $\bigcap_{i=1}^s U_i=\emptyset$, we have that $$\bigcap_{i=1}^s \bigcap_{\substack{j\leq m-k, \\ j>m-k+l}}^{m}  E(a_i^j)= \emptyset.$$ 

Finally, for $v \in V$ there exists $g \in \{0,1,2\}$ and $i, l \in \{1, \dots,  s\}$ such that $|C_g(v) \cap U_i| \neq | C_g(v) \cap U_l|$ thus $ a_l^j \neq a_i^j$ where  $ v = \iota (E(a_i^j))$.

This means that the factors $a_i^1, \dots, a_i^m$ satisfy the hypothesis of Lemma \ref{bueno} for $ 1 \leq i \leq s$.
\end{proof}
 
\begin{coro} 
 The bound obtained in Theorem \ref{mayorj} is sharp
 \end{coro}
 \begin{proof}
 By Lemma \ref{cota}, $TC_s (U\mathcal{D}^nT) \leq s(h dim (U\mathcal{D}^nT))$ and notice that since $n= 2(m-\lfloor \frac{k}{s} \rfloor +l)+ \varepsilon$ with $ l < \lceil \frac{k}{s} \rceil$,  we have $n < 2m$ and $h dim (U\mathcal{D}^nT)= \frac{n}{2}$ thus $TC_s (U\mathcal{D}^nT) \leq s\frac{n}{2}.$ 
 \end{proof}
 
\begin{ejem}\label{ejgris}
Take the tree $T$ of Figure \ref{gris} and assume every edge is sufficiently subdivided for $n=14$. Since $m= 8$, $k=7$, $l=1$ and $s=3$,  we can see in Figure \ref{gris} the set $V$ consisting of the black vertex, the set $U_1 \cap U_2$ consisting of light gray vertices, the set $U_2 \cap U_3$ consisting of gray vertices and the set $U_3 \cap U_1$ consisting of dark gray vertices.
By Theorem \ref{mayorj}, $TC_3(U\mathcal{D}^{14}T) \geq 3(8-2+1)=21$.
\end{ejem}

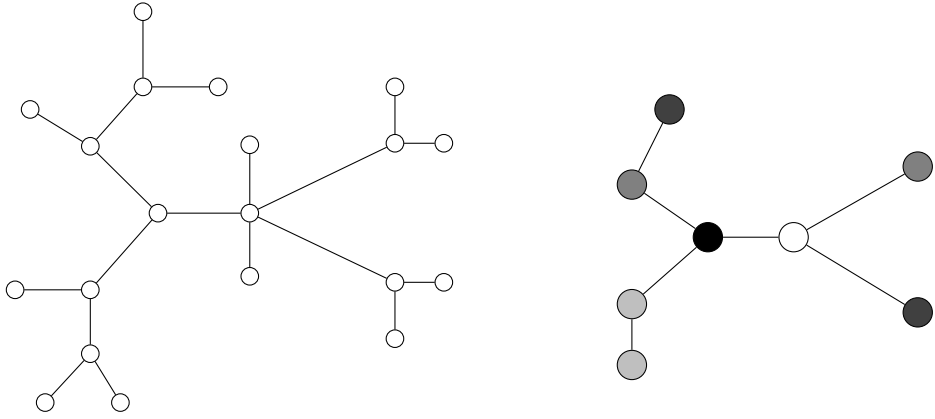
\begin{figure}
\centering
\begin{tikzpicture}[every node/.style={circle, draw}, scale=1.0,
rotate = 180, xscale = -1]

\node[scale=.6] (1) at ( 2.2, 2.68) {};
\node[scale=.6] (2) at ( 1.3, 1.79) {};
\node[scale=.6] (3) at ( 0.5, 1.3) {};
\node[scale=.6] (5) at ( 2, 1) {};
\node[scale=.6] (05) at (2,0){};
\node[scale=.6] (50) at (3,1){};
\node[scale=.6] (6) at ( 1.3, 3.7) {};
\node[scale=.6] (7) at ( 0.3, 3.7) {};
\node[scale=.6] (8) at ( 1.3, 4.55) {};
\node[scale=.6] (9) at ( 0.7, 5.2) {};
\node[scale=.6] (10) at ( 1.7, 5.2) {};
\node[scale=.6] (11) at ( 3.42, 2.68) {};
\node[scale=.6] (12) at ( 3.42, 1.77) {};
\node[scale=.6] (13) at ( 3.42, 3.52) {};
\node[scale=.6] (15) at ( 5.35, 3.6) {};
\node[scale=.6] (16) at ( 5.35, 4.35) {};
\node[scale=.6] (17) at ( 6, 3.6) {};
\node[scale=.6] (18) at ( 5.35, 1.75) {};
\node[scale=.6] (19) at ( 5.35, 1) {};
\node[scale=.6] (20) at ( 6, 1.75) {};

\node[fill=darkgray] (21) at ( 9, 1.3) {};
\node[fill=gray] (22) at ( 8.5, 2.3) {};
\node[fill=black] (23) at ( 9.51, 3.0) {};
\node[fill=lightgray] (24) at ( 8.5, 3.89) {};
\node[fill=lightgray] (25) at ( 8.5, 4.7) {};
\node (26) at ( 10.65, 3.0) {};
\node[fill=gray] (28) at ( 12.3, 2.06) {};
\node (29)[fill= darkgray] at ( 12.3, 4) {};

\draw (5)--(05);
\draw (5)--(50);
\draw (2) -- (1);
\draw (3) -- (2);
\draw (5) -- (2);
\draw (6) -- (1);
\draw (7) -- (6);
\draw (8) -- (6);
\draw (9) -- (8);
\draw (10) -- (8);
\draw (11) -- (1);
\draw (11) -- (13);
\draw (12) -- (11);
\draw (19) -- (18);
\draw (20) -- (18);
\draw (18) -- (11);
\draw (15) -- (11);
\draw (16) -- (15);
\draw (17) -- (15);
\draw (23) -- (26);
\draw (24) -- (23);
\draw (25) -- (24);
\draw (22) -- (23);
\draw (21) -- (22);
\draw (26) -- (29);
\draw (28) -- (26);

\end{tikzpicture}
\caption{The tree $T$ and a colouring of the vertices of $F(T)$. }
\label{gris}
\end{figure}

Let $F(T)$ denote the tree obtained from $T$ by removing all leaves and smoothing all bivalent vertices.

\begin{defi}
Given a tree $T$, let $h_s(T)$ denote the minimal integer such that there exist $s$ sets consisting of vertices which have degree three, $V_1, \dots, V_s$ such that:
\begin{itemize}
\item For every set $ |V_i| \leq h_s(T)$ for $ 1 \leq i \leq s$.
\item Every vertex that has degree one in $F(T)$ and degree three in $T$ belongs to a set $V_i$ with $1 \leq i\leq s$.
\item For every vertex $ x $ of degree greater than one in $F(T)$ and degree three in $T$, there exists $ j \leq d(x)-1$ and at least two indices $ i_1, i_2$ such that $$ |C_j(x) \cap V_{i_1} | \neq | C_j(x) \cap V_{i_2}|.$$
\end{itemize}
\end{defi}

Let $V$ be the set of vertices that have degree three in $T$ and degree one in $F(T)$. Notice that $h_s(T)$ always exists, since we can put $V_i =V$ for $i = 1, \dots, s-1$ and $ V_s = \emptyset$ so that for every essential vertex $x$ in $F(T)$ having degree three in $T$, we have that there exists $j$ such that $0=|C_j(x) \cap V_s| \neq |C_j (v) \cap V_i|$ for some $1 \leq i <s$. Obviously this process is not optimal, since in  this case the sets $V_i = V$ are very large.

\begin{teo}\label{tercero}
Let $n \geq 2m+h_s(T)$. Then $TC_s(U\mathcal{D}^nT) \geq sm$.
\end{teo}
\begin{proof}
Assume first that $n = 2m+h_s(T)$.
Let $x_1, \dots, x_{m-k}$ be the vertices of degree greater than three and let $y_{1}, \dots, y_k$ be the vertices of degree three.  Assume that $w_i, z_i, v_i \in N(x_i)$ for $i= 1, \dots m-k$, with $x_i<w_i<z_i<v_i$.
Let $U_i = \{ y_i+2: y_i \in V_i\}$ for $ 1 \leq i \leq s$. 

Then $$c_i= \{ (x_1,v_1), z_1, (x_2,v_2),z_2, \dots, (x_{m-k},v_{m-k}), z_{m-k},B(y_1), \dots, B(y_k), U_i\}$$ for $i$ even with $ 1< i \leq s$ and 
 $$c_j= \{ (x_1,z_1), w_1, (x_2,z_2),w_2, \dots, (x_{m-k},z_{m-k}), w_{m-k},B(y_1), \dots, B(y_k), U_j \}$$ for $j$ odd with $1 \leq j \leq s$ are $s$ m-cells.  If for some $1 \leq i \leq s$ we have that $|V_i| < h_s(T)$ then we can add to the cell $c_i$, vertices blocked at the origin. 
By Theorem \ref{fact} every cell $c_i$ for $1 \leq i \leq s$ is the strong interaction product of $m$ 1-cells $a_i^1, \dots, a_i^m$. Let $a_i^j$ be the factor which contains the edge-ingredient $(x_j, v_j) $ or $(x_j, z_j)$ (if $i$ is even or odd respectively)  for $1 \leq j \leq m-k$, and let $a_i^{l+m-k}$ be the factor which contains $B(y_{l})$ for $ 1 \leq l \leq k$

We have that $E(a_i^j) = E(a_l^j)$ for $m-k < j \leq m$ and $\iota (E(a_i^j))= y_j$ with $y_j$ a vertex of degree three. Notice also that $$\bigcap_{i=1}^s \bigcap_{j=1}^{m-k} E(a_i^j) = \bigcap_{i=1}^s (x_i, v_i) \cap (x_i, z_i) =\emptyset $$.

Finally consider the cells $a_i^j$ for $1 \leq i \leq s$ which contain $B(y_j)$. Since $y_j$ is a vertex of degree three, there exist $i,l \in \{1, \dots, s\}$ and $ g \in \{0,1,2\}$ such that $|C_g (y_j)\cap V_i| \neq  |C_g(y_j) \cap V_l|$ thus $ a_i^j \neq  a_l^j$.
 
Thus the factors $a_i^1, \dots, a_i^m$ satisfy the hypotheisis of Lemma \ref{bueno} $1 \leq i \leq s$ and hence $TC_s(U\mathcal{D}^nT) \geq sm$.

If $n > 2m+h_s(T)$ we can add to every cell $c_i \in U\mathcal{D}^{2m+h_s(T)}T$, $n-(2m +h_s(T))$ blocked vertices at the origin to obtain $c'_i \in U\mathcal{D}^nT. $
\end{proof}
 
\begin{coro}
The bound in Theorem \ref{tercero} is sharp.
\end{coro}
\begin{proof}
By Lemma \ref{cota}, $TC_s (U\mathcal{D}^nT) \leq s(h dim (U\mathcal{D}^nT))$ and notice that since $n \geq 2m+h_s(T)$, $h dim (U\mathcal{D}^nT)= m$ thus $TC_s (U\mathcal{D}^nT) \leq sm.$
\end{proof}
 
\section{Comparison with Schreirer's results} 
In this section we are going to compare our results with schreirer's results for the case when $s=2$. Since he uses very different tools, we must first introduce some definitions and notation.
We shall continue to assume that our tree $T$ is sufficiently subdivided (unless otherwise stated) and embedded in the plane.
Recall that an arc in $T$ is a subspace homeomorphic to a non-trivial closed interval. Given a finite collection
of oriented arcs $\{A_i \}^k{i=1} \in T$, and a vertex $v \in T$ of degree $d$, we will define
integers $\eta_0 (v), \eta_1 (v), \dots , \eta_{d-1} (v)$ as follows. First for $i= 1,2, \dots k$ and $j = 1, 2, \dots , d-1$, if $v$ lies on the arc $A_i$ and $A_i$ intersects the interior of the edge $e_j$ incident to $v$ on $v$-direction $j$, then $\eta_{j,i}(v)=1$ if $A_i$ is oriented towards $v$ on $e_j$ and $\eta_{j,i}(v)=-1$ if $A_i$ is oriented away from $v$ on $e_j$. If $v$ does not fall on $A_i$ or $A_i$ does not intersect the interior of $e_j$ then $\eta_{j,i}(v) =0$. Then let $$\eta_j (v) = \sum_{i=1}^k \eta_{j,i} (v).$$
We are now ready for the definition of an allowable collection of arcs.
\begin{defi}
Let $A=\{A_i\}_{i=1}^k $ be a collection of oriented arcs in $T$ and $U$ be a set of vertices of $T$. The collection $A$ is said to be allowable for $U$ if every vertex $v \in U$ of degree $d$ has the property that $v$ is not an endpoint of any $A_i$ and at least one of $\eta_0 (v), \eta_1 (v), \dots , \eta_{d-1} (v)$ is non zero.
\end{defi}
 Schreirer states the following theorem.
 \begin{teo}\cite{S} \label{sch}
 Let $T$ be a tree with $m$ essential vertices.
 \begin{enumerate}
 \item Let $p$ be the smallest integer such that there is a collection of oriented arcs $\{A_i\}_{i=1}^p$ which is allowable for the collection of all vertices of degree 3 in $T$. If there are no vertices of degree three, let $p=0$. Let $n \geq 2m+p$ be an integer. Then $TC(U\mathcal{D}^n(T))= 2m$
 \item Let $n = 2q+ \epsilon < 2m$, with $\epsilon \in \{0,1\}$ and $q \geq 1$, let $(m-k)$ be the number of vertices of degree greater than 3, and let $k$ be the number of vertices of degree three. Suppose one of these hold:
 \begin{itemize}
 \item[(a)] $k \geq 2(q-(m-k))$
 \item[(b)] \begin{enumerate}
 \item[(i)] $k <2(q-(m-k)), \epsilon=0$, and there is some $p \geq 1$ such that there exists a collection of oriented arcs $\{A_i\}_{i=1}^p$ with the following properties:
 \begin{itemize}
 \item[(A)] The endpoints of each $A_l$ are (distinct) essential vertices, neither of which is an endpoint of any other $A_l'$,
 \item[(B)] There are $r \leq m-k $ vertices of degree greater than 3 which are not the endpoints of any $A_l$,
 \item[(C)] There is a collection $U$ of vertices of degree 3-vertices, with $|U| \geq q-r-p$ such that $\{A_i\}_{i=1}^p$ is allowable for $U$.
 \end{itemize}
 \item[(ii)] $k< 2(q-(m-k)), \epsilon=1$, and there is an arc $A_0$ whose endpoints have no restrictions and whose interior includes a collection $W'$ of $s'\leq q$ distinct vertices of degree 3, and if $s<q-(m-k)$ there are arcs $A_1, \dots , A_p$, as above whose endpoints are also not vertices in $W'$, and there is an other collection of degree-3 vertices, $W$, such that $W \cap W'=\emptyset, |W|\geq q-r-p-s$ and $\{A_i\}_{i=1}^p$ is allowable for $W$, where $r$ is as above.
 \end{enumerate}
 \end{itemize}
 \end{enumerate}
 Then $TC(U\mathcal{D}^nT )= 2q$.
 \end{teo}
 
We are going to focus first on case $2(a)$ and compare it to Theorem \ref{par,s}. Notice first that in Theorem \ref{par,s}, we require that $n = 2(m-k+l) + \epsilon$ where $k$ is the amount of vertices of degree three, $m-k$ is the amount of vertices of degree greater than three and $l \leq \frac{k}{2}$. Subsituting the hypothesis of Theorem \ref{par,s} $(q=m-k+l)$ in the innequality $k \geq 2(q-(m-k))$ of Theorem \ref{sch} $2(a)$, we obtain $ k \geq 2l$, thus both Theorem  \ref{par,s} and Theorem \ref{sch} $2(a)$ require the same conditions over $n$ with respect to the amount of vertices of degree three and the amount of vertices of degree greater than three.
Since both theorems do not have any aditional hypotheisis, we conclude that they are equivalent.

\

Consider now case $(1)$ of Theorem \ref{sch}. We are going to show that it is equivalent to Theorem \ref{tercero}.
We need to prove that the smallest amount of oriented arcs which are allowable for the set of degree-3 vertices is precisely $h_2(T)$.
 
 \begin{prop}\label{eq}
 Let $p$ be the smallest integer such that there is a collection of oriented arcs $\{A_i\}_{i=1}^p$ which is allowable for the collection of all vertices of degree 3 in $T$.
Then $p = h_2(T)$. 
 \end{prop}
 \begin{proof}
 Assume first that the hypothesis of Theorem \ref{tercero} hold.
 Let $U$ be the set of vertices of degree three in $T$. There exist $V_1, V_2 \subset U$ such that $|V_i| = h_2(T)$ for $i=1,2$ and all the leaves of $F(T)$ are contained in $V_1 \cup V_2$. Assume $V_1=\{x_1, \dots, x_{h_2(T)}\}$ and $V_2=\{y_1, \dots, y_{h_2(T)}\}$, and let $Z_i$ be the unique $x_iy_i$-path in $T$. Denote by $d_{F(T)} (v)$ the degree of a vertex $v$ in $F(T)$. We are going to extend these paths as follows, for $1 \leq i \leq h_2(T)$. To make the notation easier, we will assume for now that the tree is not subdivided, this is, it has no vertices of degree two.
If $d_{F(T)}(x_i)<3$ , let $a_i$ denote a leaf in $T$ adjacent to $x_i$, and similarly, if $d_{F(T)}(y_i)<3$, let $b_i$ denote a leaf in $T$ adjacent to $y_i$. If $x_i$ is an essential vertex in $F(T)$, let $G_i$ be a path with endpoints $x_i$ and $g_i$ with $g_i$ either a leaf in $T$ or a vertex of degree greater than three in $T$ such that $g_i \neq a_j, b_j$ for $j \neq i$.
If $y_i$ is an essential vertex in $F(T)$, let $D_i$ be a path with endpoints $y_i$ and $d_i$ with $d_i$ either a leaf in $T$ or a vertex of degree greater than three in $T$ such that $d_i \neq a_j, b_j$ for $i \neq j$.  Then we define the set of oriented arcs as follows:
$$ A_i = \left\{ \begin{array}{ccc} (a_i, x_i) \cup Z_i \cup (y_i, b_i) & \mbox{ if } & d_{F(T)}(x_i), d_{F(T)}(y_i) \in \{1,2\}\\
(a_i, x_i) \cup Z_i \cup D_i & \mbox{ if } & d_{F(T)}(x_i)<3, d_{F(T)}(y_i) \geq 3 \\
G_i \cup Z_i \cup (y_i, b_i) & \mbox{ if } & d_{F(T)}(x_i)\geq 3, d_{F(T)}(y_i) <3 \\
G_i \cup Z_i \cup D_i & \mbox{ if } & d_{F(T)}(x_i)\geq 3, d_{F(T)}(y_i) \geq 3 \\
\end{array} \right. $$
The orientation of these arcs is from $ a_i$ or $g_i$ to $b_i$ or $d_i$. Notice that if $x_i$ is such that $d_{F(T)}(x_i)<3$ and $(x_i, a_i )= e_j$ for some $j \in \{0, \dots, d_T(x_i)-1\}$ then $\eta_j (x_i) =1$ and analogusly we have that $\eta_j (y_i) = -1$ if $d_{F(T)}(y_i)<3$. For $v \in U$ an essential vertex of $F(T)$, by the definition of $h_2(T)$ there exists $j \in \{0,1,2\}$ such that $|C_j(x) \cap V_{1} | \neq | C_j(x) \cap V_{2}|$ thus $\eta_j (v)\neq 0$. This means that the set of oriented arcs $\{A_i\}_{i=i}^{h_2(T)}$ is allowable for $U$ and $p \leq h_2(T)$.
 
 Now assume there exists an allowable collection of arcs $\{A_i\}_{i=1}^p$ for $U$. We are going to trim every arc just enough so that its end points are now vertices of degree three $x_i$ and $y_i$. Let $V_1 = \{x_1, \dots, x_p\}$ and $V_2=\{y_1, \dots, y_p\}$. Since every vertex of $U$ belonged to at least one arc $A_i$, the leaves of $F(T)$ must be endpoins of the trimmed arcs, thus belong to $V_1$ or $V_2$. Finally consider $v \in U$ an essential vertex of $F(T)$. If $|C_j(v) \cap V_1| =|C_j(v) \cap V_2| $ for every $j \in \{0,1,2\}$ then $\eta_j (v)=0$ for $j =0,1,2$ which is a contradiction (see for example Figure \ref{arcos}). Thus $h_2(T) \leq p$.
 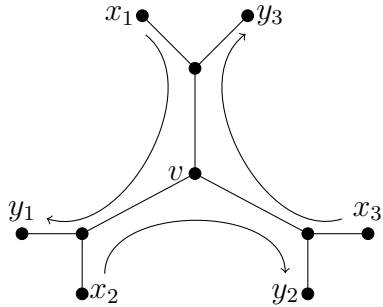
\begin{figure}[h!]
   \centering
 \begin{tikzpicture}
\node[circle, draw, scale=.4, fill=black] (2) at (1.5,-.4){};
\node[circle, draw, scale=.4, fill=black] (4) at (1.5,1){};
\node (x11) at (1,1){};
\node (y33) at (2,1){};

\node[circle, draw, scale=.4, fill=black] (04) at (.8,1.7){};
\node (x1) at (0.5,1.7){$x_1$};
\node[circle, draw, scale=.4, fill=black] (40) at (2.2,1.7){};
\node (y3) at (2.5, 1.7){$y_3$};
\node[circle, draw, scale=.4, fill=black] (5) at (3,-1.2){};
\node[circle, draw, scale=.4, fill=black] (05) at (3,-2){};
\node (y2) at (2.7,-2) {$y_2$};
\node[circle, draw, scale=.4, fill=black] (50) at (3.8,-1.2){};
\node (x3) at (3.8, -0.9){$x_3$};
\node[circle, draw, scale=.4, fill=black] (6) at (0,-1.2){};
\node (y11) at (0, -0,7){};
\node[circle, draw, scale=.4, fill=black] (60) at (0,-2){};
\node (y1) at (-0.8,-0.9) {$y_1$};
\node[circle, draw, scale=.4, fill=black] (06) at (-0.8,-1.2){};
\node (x2) at (0.3, -2){$x_2$};
\draw (x1)[->, bend left=80] to (y1); 
\draw (x2)[->, bend left=90] to (y2); 
\draw (x3)[->, bend left=80] to (y3); 
\draw (5)--(50);
\draw (5)--(05);
\draw (6)--(60);
\draw (6)--(06);
\draw (40)--(4);
\draw (04)--(4);
\draw (2)--(6);
\draw (2)--(5);
\draw (2)node[left]{$v$}--(4);
\end{tikzpicture}
\caption{An example of the contradiction in the proof of Proposition \ref{eq}}\label{arcos}
\end{figure}
\end{proof}
 This means that Theorem \ref{tercero} and Theorem \ref{sch} (1) are equivalent. Thus we only have one case remaining.
 
\begin{prop}
The hypothesis of Theorem \ref{mayorj} imply the hypothesis of Theorem \ref{sch} (2b).
\end{prop}
\begin{proof}
Assume first that $\epsilon=0$.
 Let $l' = \lfloor  \frac{k}{2} \rfloor -l>0$, since $ l >0$ we have $\lfloor  \frac{k}{2} \rfloor -l < \lfloor  \frac{k}{2} \rfloor$ thus $ l' < \frac{k}{2}$ and $2q = 2m-2l' > 2m-k$ which means that $k> 2m-2q$ and therefore $k= -k +2k <2q-2m+2k = 2(q-(m-k))$.  
 
We know there are $U_1$ and $U_2$ two sets of vertices of degree three such that $|U_i| = \lfloor \frac{k}{2} \rfloor =p$ and a set $V$ of $l$ vertices of degree three, disjoint from $U_1$ and $U_2$.
Let $U_1 = \{x_1, \dots, x_p\}$ and $U_2 =\{y_1, \dots, y_p\}$. We are going to construct arcs in a way similar to the proof of \ref{eq}, but this time we do want the vertices of $U_1$ and $U_2$ to be endpoints of the arcs as long as they are distinct. 

Consider, for every $i \in \{1,\dots, p\}$, the unique $x_i,y_i$-path in $T$. Since in $U_1 \cup U_2$ there are only $k-l$ distinct vertices, we must extend $l$ paths (or less paths if we extend some paths on both sides) to a vertex of degree higher than tree to obtain the arcs $\{A_i\}_{i=1}^p$ such that they have disctinct essential vertices as endpoints. Then $\{A_i\}_{i=1}^p$ is allowable for the set $V$ of $l$ vertices of degree three, and the amount of vertices of degree higher than three which are not an endpoint of an arc is $r= m-k-l$ thus $q-r-p = m- \lfloor \frac{k}{2} \rfloor +l -(m-k-l) - \lfloor \frac{k}{2} \rfloor = k - 2\lfloor \frac{k}{2} \rfloor \leq l$. This means that the hypothesis of Theorem \ref{sch} (2a) hold.
\end{proof}
 
In conclusion, a generalization of Schreirer's results for $TC_s$ in a more geometric way, would be to cover the set of vertices of degree three, with an "allowable" collection of $s$-stars instead of arcs, where an $s$-star is a subdivision of the complete bipartite graph $K_{1,s}$.  The problem of this, is that we can not give the stars an orientation like the arcs and therefore canot define the numbers $\eta_j.$

\begin{ejem} 
Recall the tree $T$ from Figure \ref{gris}. Then for $s=3$ we can cover it with the three 3-stars of Figure \ref{estrellas}.
\begin{figure}[h!]
\centering
\begin{tikzpicture}[every node/.style={circle, draw}, scale=1.0,
rotate = 180, xscale = -1]

\node[scale=.6] (1) at ( 2.2, 2.68) {};
\node[scale=.6] (2) at ( 1.3, 1.79) {};
\node[scale=.6] (4) at ( 2, 1) {};
\node[scale=.6] (5) at (2,0){};
\node[scale=.6] (7) at ( 1.3, 3.7) {};
\node[scale=.6] (9) at ( 1.3, 4.55) {};
\node[scale=.6] (11) at ( 1.7, 5.2) {};
\node[scale=.6] (12) at ( 3.42, 2.68) {};
\node[scale=.6] (18) at ( 5.35, 1.75) {};
\node[scale=.6] (19) at ( 5.35, 1) {};
 
\draw (4)--(5);
\draw (2) -- (1);
\draw (4) -- (2);
\draw (7) -- (1);
\draw (9) -- (7);
\draw (11) -- (9);
\draw (12) -- (1);
\draw (19) -- (18);
\draw (18) -- (12);

\node[scale=.6] (1x) at ( 8.2, 2.68) {};
\node[scale=.6] (2x) at ( 7.3, 1.79) {};
\node[scale=.6] (3x) at ( 6.5, 1.3) {};
\node[scale=.6] (7x) at ( 7.3, 3.7) {};
\node[scale=.6] (9x) at ( 7.3, 4.55) {};
\node[scale=.6] (10x) at ( 6.7, 5.2) {};
\node[scale=.6] (12x) at ( 9.42, 2.68) {};
\node[scale=.6] (15x) at ( 11.35, 3.6) {};
\node[scale=.6] (16x) at ( 11.35, 4.35) {};

\draw (2x) -- (1x);
\draw (3x) -- (2x);
\draw (9x) -- (10x);
\draw (7x) -- (1x);
\draw (7x) -- (9x);
\draw (16x) -- (15x);
\draw (1x) -- (12x);
\draw (12x) -- (15x);

\node[scale=.6] (1xy) at ( 12.2, 2.68) {};
\node[scale=.6] (12xy) at ( 13.42, 2.68) {};
\node[scale=.6] (2xy) at ( 11.3, 1.79) {};
\node[scale=.6] (4xy) at ( 12, 1) {};
\node[scale=.6] (6xy) at (13,1){};
\node[scale=.6] (15xy) at ( 15.35, 3.6) {};
\node[scale=.6] (16xy) at ( 16.2, 3.6) {};
\node[scale=.6] (18xy) at ( 15.35, 1.75) {};
\node[scale=.6] (20xy) at ( 16, 1.75) {};
\draw (2xy) -- (1xy);
\draw (1xy) -- (12xy);
\draw (12xy) -- (15xy);
\draw (2xy) -- (4xy);
\draw (12xy) -- (18xy);
\draw (18xy) -- (20xy);
\draw (15xy) -- (16xy);
\draw (4xy) -- (6xy);

\end{tikzpicture}
\caption{The tree $3$-stars covering the tree from Example \ref{ejgris}. }\label{estrellas}
\end{figure}
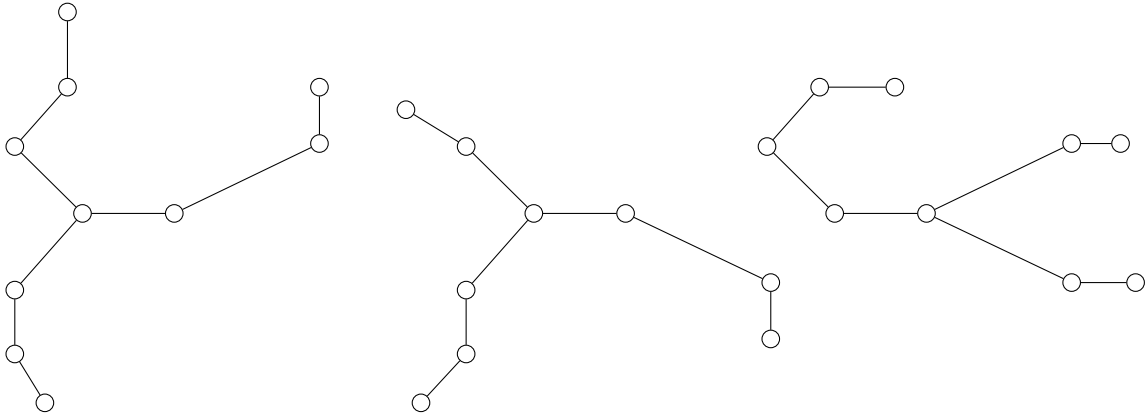 
\end{ejem}

\end{document}